\documentclass[reqno,a4paper, 11pt]{amsart}
\usepackage[english]{babel}
\usepackage[utf8]{inputenc}

\usepackage[a4paper=true,pdfpagelabels]{hyperref}
\usepackage{tikz}
\usepackage{graphicx}

\usepackage{amsfonts,epsfig}
\usepackage{latexsym}
\usepackage{mathabx}
\usepackage{amsmath}
\usepackage{amssymb}

\usepackage{color}
\def\a{\alpha}       \def\b{\beta}        
            \def\om{\omega}
              
                  \def\z{\zeta}

\def\e{\varepsilon}

\def\D{{\mathbb D}}  \def\T{{\mathbb T}}
\def\C{{\mathbb C}}  \def\N{{\mathbb N}}
\def\Z{{\mathbb Z}}

\def\R{\mathbb R}

\def\({\left(}       \def\){\right)}

\newcommand{\DD}{\widehat{\mathcal{D}}}
\newcommand{\Dd}{\widecheck{\mathcal{D}}}

\newcommand{\DDD}{\mathcal{D}}
\def\omg{\widehat{\omega}}

\def\diam{\mathord{\rm diam}}

\renewcommand{\H}{\mathcal{H}}

\newtheorem{theorem}{Theorem}[section]
\newtheorem{lemma}[theorem]{Lemma}
\newtheorem{proposition}[theorem]{Proposition}

\newtheorem{lettertheorem}{Theorem}
\newtheorem{letterlemma}[lettertheorem]{Lemma}

\theoremstyle{definition}

\theoremstyle{remark}

\numberwithin{equation}{section}

\newenvironment{Prf}{\noindent{\emph{Proof of}}}
{\hfill$\Box$ }

\begin{document}

\title[Maximal theorems for weighted analytic tent and mixed norm spaces]{Maximal theorems for weighted analytic tent and mixed norm spaces}


\author[T. Aguilar-Hern\'andez ]{Tanaus\'u Aguilar-Hern\'andez}
\address{Departamento de Matem\'atica Aplicada, Universidad de M\'alaga, Campus de Teatinos, 29071
	M\'alaga, Spain}
\email{taguilar@uma.es}

\author[A. Mas]{Alejandro Mas}
\address{Departamento de Análisis Matemático, Universidad de Valencia, 46100 Burjassot, Spain}\email{alejandro.mas@uv.es}

\author[J. A.  Pel\'aez]{Jos\'e \'Angel Pel\'aez}
\address{Departamento de An\'alisis Matem\'atico, Universidad de M\'alaga, Campus de
Teatinos, 29071 M\'alaga, Spain} \email{japelaez@uma.es}

\author[J. R\"atty\"a ]{Jouni R\"atty\"a}
\address{University of Eastern Finland\\
Department of Physics and Mathematics\\
P.O.Box 111\\FI-80101 Joensuu\\
Finland}
\email{jouni.rattya@uef.fi}

\thanks{The first author is supported by Ministerio de Ciencia e Innovaci\'on, Spain, project  project PID2022-136320NB-I00.
The rest of the  authors are supported in part by Ministerio de Ciencia e Innovaci\'on, Spain, project PID2022-136619NB-I00.
The last two authors  are supported in part by   Academy of Finland 356029 and 
La Junta de Andaluc{\'i}a, project FQM210.}


\maketitle

\begin{abstract}
Let $\omega$ be a radial weight,  $0<p,q<\infty$ and $\Gamma(\xi)=\left\{z\in\D:|\arg z-\arg\xi|<(|\xi|-|z|)\right\}$ for $\xi\in\overline{\mathbb{D}}$ . The average radial integrability space $L^q_p(\omega)$ consists of complex-valued measurable functions $f$ on the unit disc $\mathbb{D}$ such that
		$$
		\|f\|^q_{L^q_p(\omega)}
		=\frac{1}{2\pi}
		\int_{0}^{2\pi}\left(\int_{0}^{1}|f(re^{i\theta})|^p\omega(r)r\,dr\right)^{\frac{q}{p}}d\theta <\infty,
		$$
and the tent space $T^q_p(\omega)$ is the set of those $f$ for which 
	$$
	\|f\|^q_{T_{p}^{q}(\omega)}
	=\frac{1}{2\pi}\int_{\partial{\mathbb{D}}}\left(\int_{\Gamma(\xi)}|f(z)|^p\omega(z)\frac{dA(z)}{1-|z|}\right)^{\frac{q}{p}}\,|d\xi|<\infty.
	$$

Let $\H(\D)$ denote the space of analytic functions in $\mathbb{D}$. It is shown that the non-tangential maximal operator 
	$$
	f\mapsto N(f)(\xi)=\sup_{z\in\Gamma(\xi)}|f(z)|,\quad \xi\in \D,
	$$
is bounded from $AL^q_p(\omega)=L^q_p(\omega)\cap\H(\D)$ and $AT^q_p(\omega)=T^q_p(\omega)\cap\H(\D)$ to $L^q_p(\omega)$ and $T^q_p(\omega)$, respectively. These pivotal inequalities are used to establish further results such as the density of polynomials in $AL^q_p(\omega)$ and $AT^q_p(\omega)$, and the identity $AL^q_p(\omega)=AT^q_p(\omega)$ for weights admitting a one-sided integral doubling condition. Further, it is shown that any of the Littlewood-Paley formulas
\begin{equation*} \begin{split}
 \|f\|_{AL_p^q(\omega)}&\asymp\|f^{(k)} (1-|\cdot|)^k\|_{L_p^q(\omega)}+\sum_{j=0}^{k-1}|f^{(j)}(0)|, \quad f\in \H(\D),\\
\|f\|_{AT_p^q(\omega)}&\asymp\|f^{(k)} (1-|\cdot|)^k\|_{T_p^q(\omega)}+\sum_{j=0}^{k-1}|f^{(j)}(0)|, \quad f\in \H(\D),
\end{split}\end{equation*}
holds if and only if $\omega$ admits a two-sided integral doubling condition. It is also shown that the boundedness of the classical Bergman projection $P_\gamma$, induced by the standard weight $(\gamma+1)(1-|z|^2)^{\gamma}$, on $L^q_p(\omega)$ and $T^q_p(\omega)$ with $1<q,p<\infty$ is independent of $q$, and
is described by a Bekoll\'e-Bonami type condition.
\end{abstract}

\section{Introduction}

Let $\H(\D)$ denote the space of analytic functions in the unit disc $\D=\{z\in\C:|z|<1\}$. A function $\omega:\D\to [0,\infty)$, integrable over $\D$, is called a weight. It is radial if $\omega(z)=\omega(|z|)$ for all $z\in\D$. The non-tangential approach region (cone), with vertex at $\xi\in\overline{\D}$ and of opening $0<M<\infty$, is the set 
	\begin{align*}
	\Gamma_M(\xi)&=\left\{z\in\D:|\arg z-\arg\xi|<M(|\xi|-|z|)\right\}.
	\end{align*}
For $0<p,q<\infty$ and a radial weight $\omega$, the tent space $T_p^q(\omega)$ consists of complex-valued measurable functions $f$ on $\D$ such that 
	\begin{align}\label{TentA}
	\|f\|_{T_{p}^{q}(\omega)}
	=  \left(\int_{\T}\left(\int_{\Gamma(\xi)}|f(z)|^p\omega(z)\frac{dA(z)}{1-|z|}\right)^{\frac{q}{p}}\,|d\xi|\right)^{\frac1q}<\infty,
	\end{align}
where $\T=\partial{\D}$, $\Gamma(\xi)=\Gamma_1(\xi)$ and $dA(z)=\frac{dx\,dy}{\pi}$ is the normalized Lebesgue area measure on $\D$. The analytic tent space $AT_{p}^{q}(\omega)$ is $T_{p}^{q}(\omega)\cap\H(\D)$. If $\om$ is the standard weight $(\alpha+1)(1-|z|^2)^\alpha$, then we simply write $T_{p}^{q}(\omega)=T_{p}^{q}(\alpha)$ and $AT_{p}^{q}(\omega)=AT_{p}^{q}(\alpha)$. 

The concept of tent spaces were introduced by Coifman, Meyer and Stein \cite{CMS}, and these spaces have become a primordial tool in 
the operator and function theory on spaces of analytic functions. In fact, different versions of weighted tent spaces of analytic  
functions have been considered by several authors during the last decades, and they are naturally linked in several ways with classical function spaces \cite{AGJMAA2022,CV,Lu90,OFJFA1997,PelRatMathAnn,Perala2018}. One of the most used facts of this fashion is the Calderon's area theorem which asserts that $f$ belongs to the Hardy space $H^q$ if and only if $\Delta|f|^p\in T^{\frac{q}{p}}_1(1)$ \cite[Theorem~7.4]{Pabook2}. 

For $0<p,q<\infty$ and a radial weight $\omega$, the average radial integrability  space $L^q_p(\om)$ consists of complex-valued measurable functions $f$ on $\D$ such that
		\begin{equation*}
		\begin{split}
		\|f\|_{L^q_p(\om)}
		&=\left(\frac{1}{2\pi}
		\int_{0}^{2\pi}\left(\int_{0}^{1}|f(re^{i\theta})|^p\omega(r)r\,dr\right)^{\frac{q}{p}}d\theta\right)^{\frac1q}<\infty,
		\end{split}
		\end{equation*}
and 
	$$
	AL^q_p(\om)=L^q_p(\om)\cap\H(\D).
	$$
Obviously, $AT^p_p(\om)=AL^p_p(\om)$ coincides with the Bergman space $A^p_\om$ induced by $p$ and $\omega$. Throughout this paper we assume that the tail integral $\widehat{\om}(z)=\int_{|z|}^1\om(s)\,ds$
is strictly positive for all $z\in\D$, otherwise $A^p_\om=\H(\D)$.
As in the case of the tent spaces, we write $L^q_p(\alpha)$ and $AL^q_p(\alpha)$ when the inducing weight is standard. The connection between the analytic Triebel-Lizorkin spaces and the average radial integrability space $AL^q_p(\alpha)$ is known \cite{OFJFA1997}, and so is the identity $AL^q_p(\alpha)=AT^q_p(\alpha)$ \cite[p. 179]{OFJFA1997}, see also \cite[Proposition~3.1]{AGJMAA2022}. Moreover, an extensive study of essential properties of the space $AL^q_p(0)$ has been recently done in \cite{ACPMED2021,ACPJFA2022,AGJMAA2022}. However, to the best of our knowledge, the existing literature does not offer results
concerning fundamental properties and interrelationships of the spaces $AL^q_p(\om)$ and $AT^q_p(\om)$ induced by a general radial weight $\omega$.One of  the main aims of this study is to fill this gap. 
 
The first cornerstone within this weighted theory consists of proving the boundedness of the non-tangential maximal operator 
	$$
	f\mapsto N_M(f)(\xi)=\sup_{z\in\Gamma_M(\xi)}|f(z)|,\quad \xi\in\overline{\D},
	$$
from $AL^q_p(\omega)$ and $AT^q_p(\omega)$ to $L^q_p(\omega)$ and $T^q_p(\omega)$, respectively. As far as we know, the result given in Theorem~\ref{th:mainnotangencial} is new even for the standard weights.

\begin{theorem}\label{th:mainnotangencial}
Let $0<p,q,M<\infty$ and let $\omega$ be a radial weight. Then $N_M:AL^q_p(\omega)\to L^q_p(\omega)$ and $N_M:AT^q_p(\omega)\to T^q_p(\omega)$ are bounded.
\end{theorem}

The boundedness of $N_M:AT^q_p(\om)\to T^q_p(\om)$ is proved by using a version of the Fefferman-Stein vectorial maximal theorem, valid for all log-subharmonic functions and the optimal range of parameters~\cite[p.212]{Pabook2}, together with a covering of $\Gamma_M(\xi)$ induced by appropriate unions and intersections of cones with vertexes depending on $\xi$. See Section~\ref{S2} for details and a real scale illustration. The proof for $AL^q_p(\om)$ follows the same guideline but it does not involve that much geometric arguments. 

We next present several natural applications of Theorem~\ref{th:mainnotangencial}. The first one shows that each $f\in AL^q_p(\om)$ can be approximated by its dilated functions $f_\lambda(z)=f(\lambda z)$ as $\lambda\to 1^-$, and consequently, the polynomials are dense in  $AL^q_p(\om)$. The same is true for the analytic tent spaces.

\begin{theorem}\label{th:RMaprprox}
Let $0<p,q<\infty$ and let $\omega$ be a radial weight. Let $AX\in\{AL^q_p(\om),AT^q_p(\om)\}$. Then the following statements hold: 
	\begin{enumerate}
	\item[\rm(i)]
	There exists a constant $C=C(p,q)>0$ such that
	\begin{equation*}
  \|f_\lambda\|_{AX}\le C\|f\|_{AX},\quad \lambda\in\overline{\D},\quad f\in \H(\D);
	\end{equation*}
	\item[\rm(ii)] $\displaystyle\lim_{\lambda\to \zeta, \lambda\in \overline{\D}}\|f_\lambda-f_\zeta\|_{AX}=0$	for each $f\in AX$ and $\zeta\in\overline{\D}$;
	\item[\rm(iii)] Polynomials are dense in $AX(\om)$.
	\end{enumerate}
\end{theorem}
The special case  $AL^q_p(0)$ with $1\le p,q<\infty$ of Theorem~\ref{th:RMaprprox} has been recently proved in \cite[section~2]{ACPJFA2022} with different methods which do not seem to carry over to the general case.

The second application of Theorem~\ref{th:RMaprprox} concerns the  useful identity 
	\begin{equation}\label{eq:intro1}
	AT_{p}^{q}(\alpha)=AL_{p}^{q}(\alpha),
	\end{equation}
proved in \cite{OFJFA1997}, see also \cite[Proposition~3.1]{AGJMAA2022}. To pull this identity to more general setting some more definitions are in order.
A radial weight $\om$ belongs to $\DD$ if there exists $C=C(\om)>0$ such that 
	$$
	\omg(r)\le C\omg\left(\frac{1+r}{2}\right),\quad r\to1^-,
	$$
and $\om\in\Dd$ if
	$$
	\omg(r)\ge C\omg\left(1-\frac{1-r}{K}\right),\quad r\to1^-,
	$$ 
for some   $K=K(\om)>1$ and $C=C(\om)>1$.
Write $\mathcal{D}=\DD\cap\Dd$ for short, and simply say that $\omega$ is a radial doubling weight if $\om\in\DDD$. It is known that the doubling classes $\DD$ and $\DDD$ arise naturally in the operator theory related to the weighted Bergman spaces~\cite{PR19}, and hence their appearance here does not come as a surprise. Each standard radial weight obviously belongs to $\DDD$, while $\Dd\setminus\DDD$ contains exponential type weights such as
	$
	\om(r)=\exp \left(-\frac{\a}{(1-r^l)^{\b}} \right),\, \text{where}\, 0<\alpha,l,\beta<\infty.
	$
The class of rapidly increasing weights, introduced in \cite{PelRat}, lies entirely within $\DD\setminus\DDD$, and a typical example of such a weight is 
	$
	\om(z)=(1-|z|^2)^{-1}\left(\log\frac{e}{1-|z|^2}\right)^{-\alpha}$ with  $1<\alpha<\infty.
	$
To this end we emphasize that the containment in $\DD$ or $\Dd$ does not require differentiability, continuity or strict positivity. In fact, weights in these classes may vanish on a relatively large part of each outer annulus $\{z:r\le|z|<1\}$ of~$\D$. For basic properties of the aforementioned classes, concrete nontrivial examples and more, see \cite{PelSum14,PelRat,PR19} and the relevant references therein.

\begin{theorem}\label{th:samespaceintro 1}
Let $0<p,q<\infty$ and $\om\in\DD$. Then $AL^q_p(\om)=AT^q_p(\om)$ with equivalence of quasinorms.
\end{theorem}

 Theorem~\ref{th:samespaceintro 1} is sharp in the sense that if $\omega$ belongs to the class  $\mathcal{W}$ of rapidly decreasing weights, defined in Section~\ref{S2}, then the statement is in general false.
The class $\mathcal{W}$ is a large set of smooth weights and it has been widely studied \cite{Atesis,ATTY,Pau-Pelaez:JFA2010}. It contains, for example, the (iterated) exponential type weights among many others.

\begin{theorem}\label{th:counterexaple}
Let $\mu \in \mathcal{W}$. Then there exist $0<p,q<\infty$ such that $AT^q_p(\mu)\neq AL^q_p(\mu)$.
\end{theorem}

It is  worth underlining here that the identity $AL^q_p(\om)=AT^q_p(\om)$ for $\om\in\DD$   is a particular phenomenon of analytic radial integrability and tent spaces that does not remain true for the corresponding spaces of measurable functions as the following result shows.

\begin{proposition}\label{th:intro1}
Let $0<p,q<\infty$ and let $\omega$ be a radial weight. Then the following statements hold:
	\begin{enumerate}
	\item[\rm(i)] If $p<q$, then $L^q_p(\omega)\subsetneq T^q_p(\omega)$ with $\|f\|_{T^q_p(\omega)}\lesssim
	\|f\|_{L^q_p(\omega)}$ for all measurable functions $f$; 
	\item[\rm(ii)] $L^p_p(\omega)=L^p_\om=T^p_p(\omega)$ with equivalence of quasinorms;
	\item[\rm(iii)] If $q<p$, then $T^q_p(\omega)\subsetneq L^q_p(\omega)$ with $\|f\|_{L^q_p(\omega)}\lesssim\| f\|_{T^q_p(\omega)}$
for all measurable functions $f$.
	\end{enumerate}
\end{proposition}

Proposition~\ref{th:intro1} is likely known for experts working in the area but we give a detailed proof because we did not find it in the literature.  

In the theory of weighted spaces of analytic functions it is often useful to know if the inducing weight of the quasinorm in question can be replaced by another one which possesses strong smoothness properties. Our next result, which is strongly based in Theorem~\ref{th:mainnotangencial}, falls into this category of results and shows that, whenever $\om\in\DDD$, we may replace it in the quasinorm by the regularized weight 
	$$
	\widetilde{\omega}(z)=\frac{\widehat{\omega}(z)}{1-|z|},\quad z\in\D,
	$$
which does not have zeros. Since for $\om\in\DDD$ we have $\widetilde\om\in\DDD$, we may iterate the result and thus assume that the inducing weight is, for instance, differentiable and non-vanishing.

\begin{theorem}\label{th:samespace2}
Let $0<p,q<\infty$ and let $\om$ be a radial weight. Then the following statements are equivalent:
	\begin{itemize}
	\item[\rm(i)] $\om\in\DDD$;
	\item[\rm(ii)] $AL^q_p(\om)=AL^q_p(\widetilde{\om})$ with equivalence of quasinorms;
	\item[\rm(iii)] $AT^q_p(\om)=AT^q_p(\widetilde{\om})$ with equivalence of quasinorms.
	\end{itemize}
In particular, if $\om\in\DDD$ then $AL^q_p(\om)=AT^q_p(\om)=AL^q_p(\widetilde{\om})=AT^q_p(\widetilde{\om})$ with equivalence of quasinorms.
\end{theorem}

Another commonly used tool in the operator theory in spaces of analytic functions is quasinorms in terms of the iterated derivatives. In our context, the Littlewood-Paley formula
	\begin{equation}\label{eq:LPintro}
  \|f\|_{AT_p^q(\om)}\asymp\|f^{(k)} (1-|\cdot|)^k\|_{T_p^q(\om)}+\sum_{j=0}^{k-1}|f^{(j)}(0)|, \quad f\in \H(\D),
	\end{equation}
is known for the standard radial weights $\omega$ and $0<p,q<\infty$ by \cite[Theorem~2]{Perala2018}, see also \cite{ACPMED2021} and \cite{OFJFA1997} for related results. This raises the question of which properties of radial weights are determinative for \eqref{eq:LPintro} to hold. Our next result provides a neat answer to this question.

\begin{theorem}\label{thm:LPchar}
Let $0<p,q<\infty$ and $k\in\N$, and let $\omega$ be a radial weight. Then the following statements are equivalent:
	\begin{enumerate}
  \item[\rm(i)] $\omega\in\mathcal{D}$;
  \item[\rm(ii)] $\displaystyle\|f\|_{AT_p^q(\om)}\asymp\|f^{(k)}(1-|\cdot|)^k\|_{T_p^q(\om)}+\sum_{j=0}^{k-1}|f^{(j)}(0)|$ for all $f\in \H(\D)$;
  \item[\rm(iii)] $\displaystyle\|f\|_{AL_p^q(\om)}\asymp\|f^{(k)}(1-|\cdot|)^k\|_{L_p^q(\om)}+\sum_{j=0}^{k-1}|f^{(j)}(0)|$ for all $f\in \H(\D)$.
  \end{enumerate}
\end{theorem}

Theorem~\ref{thm:LPchar} is known in the case $q=p$ which corresponds to the weighted Bergman space~\cite[Theorem~5]{PR19}. The proof of Theorem~\ref{thm:LPchar} relies strongly on Theorem~\ref{th:mainnotangencial} together with  properties of radial doubling weights. We will actually prove slightly more than what is stated above. Namely, we will show that the inequality
	\begin{equation}\label{eq:introLP2}
  \|f^{(k)} (1-|\cdot|)^k\|_{T_p^q(\om)}+\sum_{j=0}^{k-1}|f^{(j)}(0)|\lesssim\|f\|_{AT_p^q(\om)},\quad f\in \H(\D),
	\end{equation} 
or its analogue for $AL_p^q(\om)$, holds if and only if $\om\in\DD$. This result extends \cite[Theorem~6]{PR19}.

We next appeal to Theorem~\ref{thm:LPchar} to describe the analytic symbols such that certain integration operators are bounded on $AT^q_p(\omega)$. To give the precise statement, for each $g\in\H(\D)$, define
	$$
	T_g(f)(z)=\int_0^z f(\zeta)g'(\zeta)\,d\zeta,\quad z\in\D,
	$$
and for $a=(a_1,a_2,\dots, a_{n-1})\in \C^{n-1}$, set
	\begin{align*}
  T_{g,a}(f)= T_{I}^{n}\left(f g^{(n)}+\sum_{k=1}^{n-1} a_{k} f^{(k)} g^{(n-k)}\right), \quad I(z)=z, z\in\D.
	\end{align*}
The generalized integration operator $T_{g,a}$ was introduced by Chalmoukis in \cite{Chalmoukis}.

\begin{theorem}\label{thm:Tga_bound}
Let $0<p,q<\infty$, $n\in \N$, $a\in \mathbb{C}^{n-1}$ and $\omega\in \mathcal{D}$. Then $T_{g,a}: AL_p^q(\omega)\rightarrow AL_p^q(\omega)$ is bounded if and only if 
$g$ belongs to the classical Bloch space $\mathcal{B}$. 
\end{theorem}

The last topic of the paper is to describe the dual spaces of $AL^q_p(\om)$ and $AT^q_p(\om)$. Our approach to this matter leads us to characterize the radial weights such that the Bergman projection
	$$
	P_\gamma(f)(z)=(\gamma+1)\int_{\D} \frac{f(\zeta)}{(1-\overline{\zeta}z)^{2+\gamma}}(1-|\z|^2)^{\gamma}\,dA(\z), \quad f\in L^1_\gamma,\quad z\in\D,
	$$
induced by the standard radial weight $(\gamma+1)(1-|\z|^2)^{\gamma}$, is bounded on $L^{q}_p(\omega)$ or $T^{q}_p(\omega)$. We also consider the maximal Bergman projection
	$$
	P^+_\gamma(f)(z)=(\gamma+1)\int_{\D} \frac{f(\zeta)}{|1-\overline{\zeta}z|^{2+\gamma}}(1-|\z|^2)^{\gamma}\,dA(\z), \quad f\in L^1_\gamma, \quad z\in\D.
	$$
For brevity, we write $\om_{x}=\int_0^1r^x\om(r)\,dr$ for the moments of a radial weight $\omega$. If $\om(z)=v_{\gamma}(z)=(\gamma+1)(1-|\z|^2)^{\gamma}$,
we will simply write $\om_x=\gamma_x$ to denote its moments.

\begin{theorem}\label{BergmanProj1}
Let $1<p,q<\infty$ and $-1<\gamma<\infty$, and let $\omega$ be a radial weight. Then the following statements are equivalent:
\begin{enumerate}
\item[\rm(i)] $P_\gamma:L^q_p(\omega)\to AL^q_p(\omega)$ is bounded;
\item[\rm(ii)]  $P^+_\gamma:L^q_p(\omega)\to L^q_p(\omega)$ is bounded;
\item[\rm(iii)] $\displaystyle D_p(\gamma,\omega)=\sup_{n\in\N\cup\{0\}}\frac{\left(\omega_{np+1}\right)^\frac1p\left(\sigma_{np'+1}\right)^\frac1{p'}}{\gamma_{2n+1}}<\infty$, where	
	$$
	\sigma=\sigma_{v_{\gamma},\om,p}=\left(\frac{v_{\gamma}}{\om^\frac1p}\right)^{p'}=\frac{v_{\gamma}^\frac{p}{p-1}}{\om^\frac1{p-1}};
	$$
\item[\rm(iv)] The Bekoll\'e-Bonami type condition 
	\begin{equation*}
	B_p(\gamma,\omega)
	=\sup\limits_{0\le r<1}\frac{\left(\int_r^1 \omega(t)t \,dt\right)^\frac1p\left(\int_r^1 \sigma(t)t\,dt\right)^\frac1{p'}}{\int_r^1  v_{\gamma}(t)t\,dt}<\infty
	\end{equation*}
holds;
\item[\rm(v)] $\left( AL^q_p(\omega)\right)^\star\simeq AL^{q'}_{p'}(\sigma)$ via the $A^2_\gamma$-pairing with equivalence of norms;
\item[\rm(vi)]  $\left( AT^q_p(\omega)\right)^\star\simeq AT^{q'}_{p'}(\sigma)$ via the $A^2_\gamma$-pairing with equivalence of norms;
\item[\rm(vii)] $P_\gamma:T^q_p(\omega)\to AT^q_p(\omega)$ is bounded.
\end{enumerate}
\end{theorem}

The boundedness of $P_0:L^q_p(0)\to AL^q_p(0)$ was recently proved in \cite[Theorem~4.3]{ACPJFA2022}, see also \cite[Proposition~2.8]{OFJFA1997}. In fact, to show that (i)$\Rightarrow$(iv) we will follow ideas from the proof of \cite[Theorem~4.3]{ACPJFA2022} to control the maximal Bergman kernel $\frac{(1-|\z|^2)^{\gamma}}{|1-\overline{\zeta}z|^{2+\gamma}}$ by an adequate discrete kernel. Among other tools used in the proof, we apply results concerning the boundedness of the H\"ormander-type maximal function
		$$
    M_{\om}(\varphi)(z)=\sup_{z\in S}\frac{1}{\om\left(S\right)}\int_{S}|\varphi(\z)|\om(\z)\,dA(\z)
    $$
where $S$ is a Carleson square. Namely, we will use the fact that $M_\om:L^p_\omega\to L^p_\omega$ is bounded for each $\om\in\DD$ \cite[Theorem~3.4]{PelSum14}, and appeal to \cite[(4.7)]{PottRegueraJFA}, see also \cite{BB}, which states that $M_{v_{\gamma}}:L^p_\om\to L^p_\om$  is bounded if and only if $\omega$ belongs to the Bek\'olle-Bonami class $B_p(\gamma,\omega)$.

\medskip
The rest of the paper is organized as follows. Theorems~\ref{th:mainnotangencial}-- \ref{th:samespace2} are proved in Section~\ref{S2}.
Section~\ref{S3} is mainly devoted to the proof of Theorem~\ref{thm:LPchar}, while Section~\ref{S4} contains the proof of Theorem~\ref{thm:Tga_bound}. Finally, Theorem~\ref{BergmanProj1} is proved in Section~\ref{S5}. 

For clarity, a word about the notation already used in this section and to be used throughout the paper. The letter $C=C(\cdot)$ will denote an absolute constant whose value depends on the parameters indicated in the parenthesis, and may change from one occurrence to another. As usual, for non-negative functions $A$ and $B$, the notation $A\lesssim B$, or equivalently $B\gtrsim A$, means that $A\le C\,B$ for some  constant $C>0$ independent of the variables involved. Further, we write $A\asymp B$ when $A\lesssim B\lesssim A$. 

\section{Maximal functions, equivalent norms and polynomial approximation}\label{S2}

We begin with the nontangential maximal operator acting on $AL^q_p(\om)$ because it serves us as a model for the tent space case which is more involved. The radial maximal function of a measurable function $\phi:\D\to\C$ is 
	\begin{equation}\label{eq:radialmaximal}
	R(\phi)(z)=\sup_{0\le r\le 1}|\phi(rz)|,\quad z\in\D.
	\end{equation}

\begin{theorem}\label{th:radialmaximal}
Let $0<p,q,M<\infty$ and let $\om$ be a radial weight. Then there exists a constant $C=C(p,q,M)>0$ such that
		\begin{equation}\label{eq:RMpqconvunif}
		\|R(f)\|_{L^q_p(\om)}\le \|N_M(f)\|_{L^q_p(\om)}\le  C \|f\|_{AL^q_p(\om)},\quad f\in \H(\D).
		\end{equation} 
\end{theorem}

\begin{proof}
The left-hand inequality in \eqref{eq:RMpqconvunif} is obvious, so we only we have to prove the right one. 
First observe that $N_M(\phi)(r\zeta)$ is a non-decreasing function of $r$ for each fixed $\zeta\in\T$. Therefore, for all $s\in(0,1)$ and $\theta\in [0,2\pi)$, we have
		\begin{equation*}
		\begin{split}
		\int_0^s( N_M(f))^p(re^{i\theta})\omega(r)r\,dr
		&=\sum_{j=0}^{n-1}\int_{\frac{js}{n}}^{\frac{(j+1)s}{n}} N_M(|f|^p)(re^{i\theta})\omega(r)r\,dr\\
		&\le\sum_{j=0}^{n-1} N_M(|f|^p)\left(\frac{(j+1)s}{n}e^{i\theta}\right)\int_{\frac{js}{n}}^{\frac{(j+1)s}{n}}\omega(r)r\,dr,\quad n\in\N.
		\end{split}
		\end{equation*}
Since the functions
		$$
		h_{j,s}(z)=\left|f\left(\frac{(j+1)s}{n}z\right)\right|^p \int_{\frac{js}{N}}^{\frac{(j+1)s}{N}}\omega(r)r\,dr,\quad j=0,\dots,n-1,
		$$
are log-subharmonic in $\D$  and continuous in $\overline{\D}$, by \cite[Theorem~7.2 on p.~212 and the comment following it]{Pabook2}, there exists a constant $C=C(p,q)>0$ such that
		\begin{equation*}
		\begin{split}
		\int_{0}^{2\pi}\left(\int_0^s( N_M(f))^p(re^{i\theta})\omega(r)r\,dr\right)^{\frac{q}{p}}\,d\theta
		&\le\int_{0}^{2\pi}\left(\sum_{j=0}^{n-1} N_M(h_{j,s})(e^{i\theta})\right)^{\frac{q}{p}}\,d\theta\\
		&\le C\int_{0}^{2\pi}\left(\sum_{j=0}^{n-1}h_{j,s}(e^{i\theta})\right)^{\frac{q}{p}}\,d\theta\\
		&=C\int_{0}^{2\pi}\left(\sum_{j=0}^{n-1}\left| f \left(\frac{(j+1)s}{n}e^{i\theta}\right)\right|^p 
		\int_{\frac{js}{n}}^{\frac{(j+1)s}{n}}\omega(r)r\,dr\right)^{\frac{q}{p}}\,d\theta.
		\end{split}
		\end{equation*}
By using that $|f|^p$ is uniformly continuous on $\overline{D(0,s)}$, it now follows that 
		$$
		\lim_{n\to\infty}\sum_{j=0}^{n-1}\left|f\left(\frac{(j+1)s}{n}e^{i\theta}\right)\right|^p
		\int_{\frac{js}{n}}^{\frac{(j+1)s}{n}}\omega(r)r\,dr
		=\int_{0}^s|f(re^{i\theta})|^p\omega(r)r\,dr.
		$$
Therefore, by the dominated convergence theorem,
		\begin{equation*}\begin{split}
		\int_{0}^{2\pi}\left(\int_0^s(N_M(f))^p(re^{i\theta})\omega(r)r\,dr\right)^{\frac{q}{p}}\,d\theta
		&\le C\int_{0}^{2\pi}\left(\int_{0}^s|f(re^{i\theta})|^p\omega(r)r\,dr\right)^{\frac{q}{p}}\,d\theta
		\le C\|f\|_{AL^q_p(\om)}^q,
		\end{split}
		\end{equation*}
from which the assertion \eqref{eq:RMpqconvunif} follows by Fatou's lemma. 
\end{proof}

Theorem~\ref{th:radialmaximal} implies that, for each fixed $0<M<\infty$, we have 
		\begin{equation}\label{AAA}
		\|f\|_{AL^q_p(\om)}\asymp\|R(f)\|_{L^q_p(\om)}\asymp\|N_M(f)\|_{L^q_p(\om)},\quad f\in \H(\D).
		\end{equation}
In order to obtain an analogue of this result for the tent spaces we will use the proof of Theorem~\ref{th:radialmaximal} as a toy model. The task we will face then is to find a suitable partition for each cone $\Gamma$ with vertex in $\D$ in terms of sets induced by suitably chosen cones induced by points in $\overline{\Gamma}$. Our cones have been chosen such that they obey convenient geometric properties and that allows us to pull the argument from radii to tents. 

\begin{theorem}\label{th:tentmaximal}
Let $0<p,q,M<\infty$ and let $\om$ be a radial weight. Then there exists a constant $C=C(p,q,M)>0$ such that
		\begin{equation}\label{eq:RMpqconvunif2}
		\|N_M(f)\|_{T^q_p(\om)}\le C \|f\|_{AT^q_p(\om)},\quad f\in \H(\D).
		\end{equation} 
\end{theorem}

\begin{proof} 
A direct calculation shows that 
	\begin{equation}
	\Gamma_M(\xi)=\bigcup_{z\in\Gamma_M(\xi)}\Gamma_M(z),\quad \xi\in\overline{\D}\setminus\{0\}.
	\end{equation}
For $\xi\in\overline{\D}\setminus\{0\}$, define $\xi^n_{0,0}=\xi$ and
	$$
	\xi^n_{j,k}
	=\xi\frac{n-j}{n}e^{iM|\xi|\frac{k}{n}},\quad j=1,\ldots,n-1,\quad-j\le k\le j,\quad n\in\N, 
	$$
and set
	$$
	F^n_{j,k}(\xi)=\Gamma_{M}(\xi_{j,k}^n)
	\setminus\left(\bigcup_{l\ge j+1,\,|i|\le l}\Gamma_{M}(\xi_{l,i}^n)\right).
	$$
Further, define $E^n_{j,0}(\xi)=F^n_{j,0}(\xi)$ and $E^n_{j,k}(\xi)=F^n_{j,k}(\xi)\setminus\cup_{|i|\le k-1}F^n_{j,i}(\xi)$ for $0<|k|\le j$ and $j=0,\ldots,n-1$. Then the sets $E^n_{j,k}(\xi)$ are pairwise disjoint such that
	\begin{equation}\label{lfdaskjlkfjahlkafjsh}
	\bigcup_{j=0}^{n-1}\bigcup_{-j\le k\le j}E^n_{j,k}(\xi)=\Gamma_M(\xi),\quad n\in\N,
	\end{equation}
and
	$$
	\max_{j=0,\ldots,n-1}\max_{-j\le k\le j}\diam\,E^n_{j,k}(\xi)\to0,\quad n\to\infty,
	$$
see Figure~\ref{fig:1} for an  illustration of the sets $E^n_{j,k}(\xi)$. 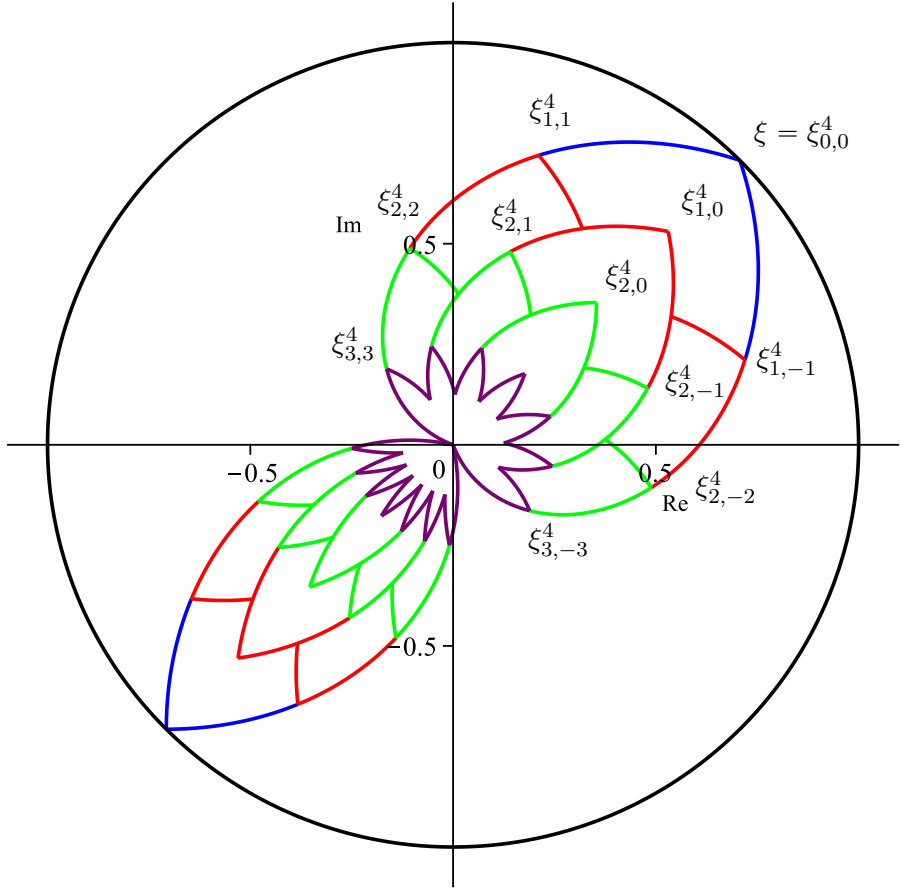
\begin{figure}[htp]
  \begin{center}
    \begin{tikzpicture}
      \draw (0, 0) node[inner sep=0] {\includegraphics[keepaspectratio,height=12cm]{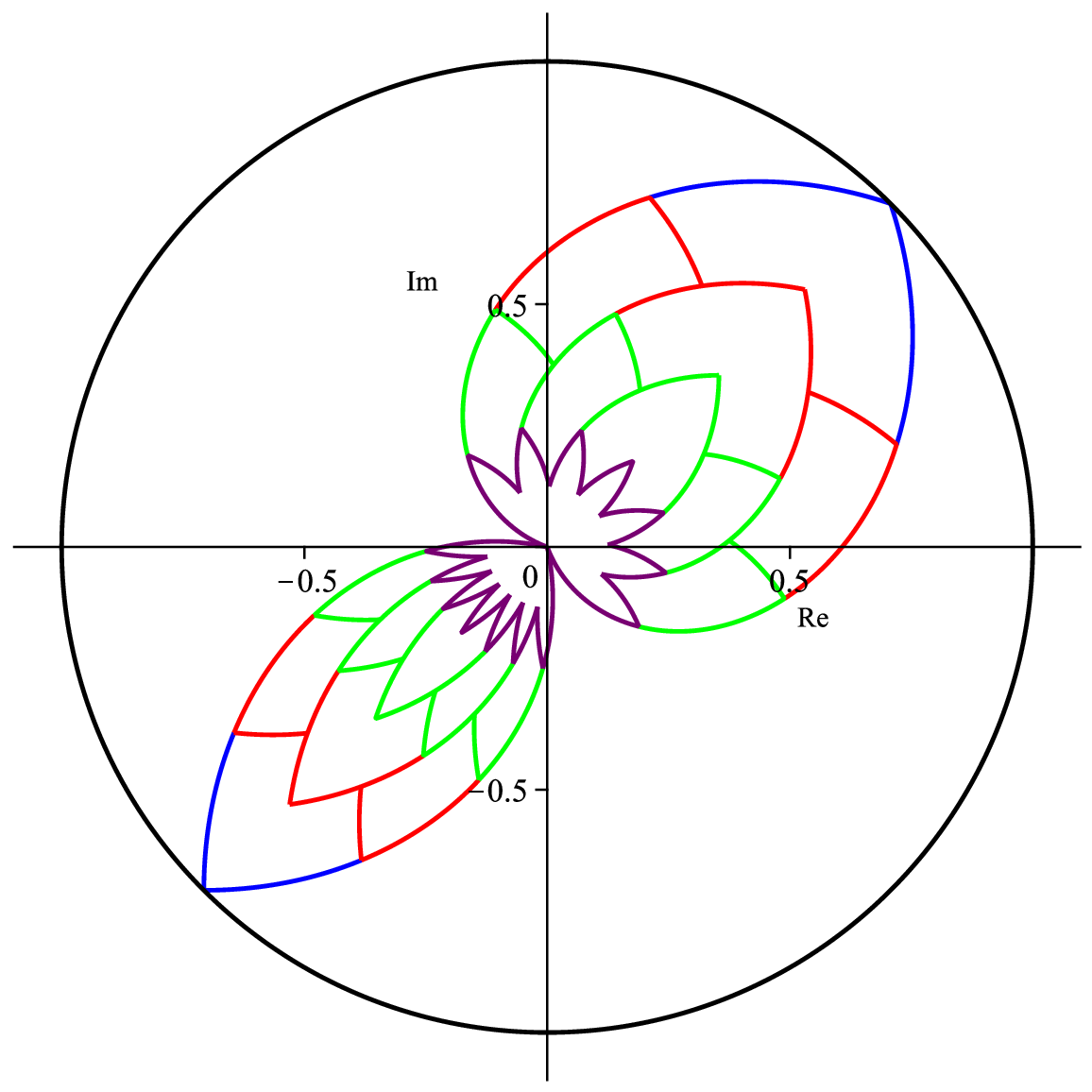}};
      \draw (4.6, 4.1) node {$\xi=\xi_{0,0}^4$};
      \draw (3.3, 3.2) node {$\xi_{1,0}^4$};
      \draw (1.3, 4.4) node {$\xi_{1,1}^4$};
      \draw (4.4, 1.1) node {$\xi_{1,-1}^4$};
      \draw (2.3, 2.2) node {$\xi_{2,0}^4$};
      \draw (0.8, 3.0) node {$\xi_{2,1}^4$};
      \draw (3.2, 0.8) node {$\xi_{2,-1}^4$};
      \draw (3.6, -0.6) node {$\xi_{2,-2}^4$};
      \draw (-0.7, 3.2) node {$\xi_{2,2}^4$};
      \draw (1.4, -1.3) node {$\xi_{3,-3}^4$};
      \draw (-1.3, 1.3) node {$\xi_{3,3}^4$};
    \end{tikzpicture}
    \caption{Two real scale illustrations of the sets $E^n_{j,k}(\xi)$ when $\xi=e^{i\left( \frac{3\pi}{4} \pm  \frac{\pi}{2} \right)}$ and $n=4$. In the case 
$\xi=e^{i\frac{\pi}{4}}$ the opening is $M=2$ while in the other case $M=1$. }
    \label{fig:1}
  \end{center} 
\end{figure} Fatou's lemma implies
	\begin{equation*}
	\begin{split}
	\|N_M(f)\|_{T^q_p(\om)}^q
	&\le\lim_{s\to1^-}\int_\T\left(\int_{\Gamma_M(s\xi)}N_M(f)^p(z)\frac{\om(z)}{1-|z|}\,dA(z)\right)^{\frac{q}{p}}|d\xi|\\
	&=\lim_{s\to1^-}\int_\T\left(\sum_{j=0}^{n-1}\sum_{-j\le k\le j}\int_{E^n_{j,k}(s\xi)}N_M(f)^p(z)\frac{\om(z)}{1-|z|}\,dA(z)\right)^{\frac{q}{p}}|d\xi|\\
	&\le\lim_{s\to1^-}\int_\T\left(\sum_{j=0}^{n-1}\sum_{-j\le k\le j}N_M(f)^p\left((s\xi)^n_{j,k}\right)\int_{E^n_{j,k}(s\xi)}\frac{\om(z)}{1-|z|}\,dA(z)\right)^{\frac{q}{p}}|d\xi|\\
	&=\lim_{s\to1^-}\int_\T\left(\sum_{j=0}^{n-1}\sum_{-j\le k\le j}N_M(h^n_{j,k})(\xi)\right)^{\frac{q}{p}}|d\xi|,
	\end{split}
	\end{equation*}
where
	$$
	h^n_{j,k}(\zeta)=\left|f\left(s^n_{j,k}\zeta\right)\right|^p\int_{E^n_{j,k}(s\xi)}\frac{\om(z)}{1-|z|}\,dA(z),\quad \zeta\in\overline{\D},
	$$
are log-subharmonic in $\D$. Now \cite[Theorem~7.2 on p.~212 and the comment following it]{Pabook2} implies
	\begin{equation*}
	\begin{split}
	\|N_M(f)\|_{T^q_p(\om)}^q
	&\lesssim\lim_{s\to1^-}\int_\T\left(\sum_{j=0}^{n-1}\sum_{-j\le k\le j}h^n_{j,k}(\xi)\right)^{\frac{q}{p}}|d\xi|,\quad n\in\N.
	\end{split}
	\end{equation*}
Since the diameter of each set $E^n_{j,k}(s\xi)$ tends to zero as $n\to\infty$, and $\Gamma(s\xi)$ expands to $\Gamma(\xi)$ as $s\to1^-$, the identity \eqref{lfdaskjlkfjahlkafjsh} yields
	\begin{equation*}
	\begin{split}
	\|N_M(f)\|_{T^q_p(\om)}^q
	&\lesssim\lim_{s\to1^-}\int_\T\left(\int_{\Gamma_M(s\xi)}|f(z)|^p\frac{\om(z)}{1-|z|}\, dA(z)\right)^{\frac{q}{p}}|d\xi|
	\lesssim\|f\|_{AT^q_p(\om)}^q,\quad f\in\H(\D).
	\end{split}
	\end{equation*}
Thus the theorem is proved.
\end{proof}

Theorem~\ref{th:tentmaximal} implies 
		\begin{equation}\label{BBB}
		\|f\|_{AT^q_p(\om)}\asymp\|R(f)\|_{T^q_p(\om)}\asymp\|N_M(f)\|_{T^q_p(\om)},\quad f\in \H(\D).
		\end{equation}
By combining this with \eqref{AAA} we obviously have Theorem~\ref{th:mainnotangencial}.	We next proceed to apply Theorem~\ref{th:mainnotangencial} to prove our other results.

\medskip

\begin{Prf}{\em{Theorem~\ref{th:RMaprprox}.}}
We will only establish the case $AX=AL^q_p(\om)$, the proof for the tent spaces readily follows by the same argument. 

(i). Let $\lambda\in\overline{\D}$. Then Theorem~\ref{th:radialmaximal} implies 
		$$ 
		\|f_\lambda\|_{AL^q_p(\om)}
		=\|f_{|\lambda|}\|_{AL^q_p(\om)}
		\le\|R(f_{|\lambda|})\|_{L^q_p(\om)}
		\le\|R(f)\|_{L^q_p(\om)}
		\lesssim\|f\|_{AL^q_p(\om)},\quad f\in \H(\D).
	$$
	
(ii). Let $f\in AL^q_p(\omega)$. Then
		$$
		\int_0^1|f(se^{i\theta})|^p\omega(s)s\,ds<\infty
		\quad\text{and}\quad
		\lim_{r\to 1^{-}}\int_r^1|f(se^{i\theta})|^p\omega(s)s\,ds=0
		$$
for almost every $\theta\in[0,2\pi]$.	Consequently, for each $\varepsilon>0$, there exists $r_0=r_0(\e)\in(0,1)$ such that
		\begin{equation}\label{eq:j30}
		\int_{0}^{2\pi}\left( \int_{r_0}^1 |f(se^{i\theta})|^p s\omega(s)\,ds\right)^{\frac{q}{p}}\,d\theta<\varepsilon
		\end{equation}
by the dominated convergence theorem. Further, by arguing as in the proof of Theorem~\ref{th:radialmaximal} we find a constant $C=C(q,p)>0$ such that
		\begin{equation}\label{eq:j31}
		\int_{0}^{2\pi}\left( \int_{r_0}^1|R(f)(se^{i\theta})|^p\omega(s)s\,ds\right)^{\frac{q}{p}}\,d\theta
		\le C\int_{0}^{2\pi}\left(\int_{r_0}^1|f(se^{i\theta})|^p\omega(s)s\,ds\right)^{\frac{q}{p}}\,d\theta.
		\end{equation}
By combining \eqref{eq:j30} and \eqref{eq:j31} we obtain
	\begin{equation*}
	\begin{split}
	&\int_{0}^{2\pi}\left(\int_{r_0}^1|(f_\lambda-f_\zeta)(se^{i\theta})|^p\omega(s)s\,ds\right)^{\frac{q}{p}}\,d\theta\\ 
	&\lesssim\int_{0}^{2\pi}\left(\int_{r_0}^1|(f_\lambda)(se^{i\theta})|^p\omega(s)s\,ds\right)^{\frac{q}{p}}\,d\theta
	+\int_{0}^{2\pi}\left(\int_{r_0}^1|f_\zeta(se^{i\theta})|^p\omega(s)s\,ds\right)^{\frac{q}{p}}\,d\theta\\
	&\le2\int_{0}^{2\pi}\left(\int_{r_0}^1|R(f)(se^{i\theta})|^p\omega(s)s\,ds\right)^{\frac{q}{p}}\,d\theta
	\le 2C\varepsilon,\quad \lambda,\zeta\in\overline{\D}.
	\end{split}
	\end{equation*}
Since $f$ is uniformly continuous in $\overline{D(0,r_0)}$, there exists $\delta=\delta(\e)>0$ such that $|f(z)-f(w)|<\varepsilon^{\frac1q}$ for all $z,w\in\overline{D(0,r_0)}$ such that $|z-w|<\delta$. Therefore, if
 $\lambda,\zeta\in\overline{\D}$ with
 $|\lambda-\zeta|<\delta$, we have
		\begin{equation*}
		\begin{split}
		\|f_\lambda-f_\zeta\|_{AL^q_p(\om)}^q
		&\lesssim\int_{0}^{2\pi}\left(\int_{0}^{r_0}|(f_\lambda-f_\zeta)(se^{i\theta})|^p\omega(s)s\,ds\right)^{\frac{q}{p}}\,d\theta\\
		&\quad+\int_{0}^{2\pi}\left(\int_{r_0}^1|(f_\lambda-f_\zeta)(se^{i\theta})|^p\omega(s)s\,ds\right)^{\frac{q}{p}}\,d\theta 			\lesssim\varepsilon,
		\end{split}
		\end{equation*}
and (ii) follows.
Since (ii) implies (iii) by standard arguments, the theorem is proved.	
\end{Prf}

	\medskip

If $\om,\nu:[0,1)\to[0,\infty)$ are integrable and satisfy $\widehat{\om}\lesssim\widehat\nu$ on $[\rho,1)$, then an integration by parts shows that  
	\begin{equation}\label{trivial}
	\int_\rho^1\varphi(r)\om(r)\,dr\lesssim\int_\rho^1\varphi(r)\nu(r)\,dr
	\end{equation}
for all non-decreasing functions $\varphi:[0,1)\to[0,\infty)$, see \cite[Lemma~8]{PGRV2024} for details. This observation serves us in several instances in the sequel.

\medskip
\begin{Prf}{\em{ Theorem~\ref{th:samespaceintro 1}.}}
Since $z\in\Gamma_M(\xi)$ if and only if $z\in\Gamma_{KM}((|z|+(1-|z|)/K)\xi)$ for all $\xi\in\T$, $0<M<\infty$ and $1<K<\infty$, we have $\Gamma\left(\xi\right)\cap D(0,r)\subset\Gamma_2\left(\frac{1+r}{2}\xi\right)$ for all $0<r<1$ and $\xi\in\T$. Therefore 
	\begin{equation*}
	\begin{split}
	\|f\|_{AT^q_p(\om)}^q
	&\lesssim\int_0^{2\pi}\left(\int_0^1N_2(f)^p\left(\frac{1+r}{2}e^{i\theta}\right)\om(r)r\,dr\right)^\frac{q}{p}d\theta\\
	&=\int_0^{2\pi}\left(\int_{\frac12}^1N_2(f)^p\left(te^{i\theta}\right)\om(2t-1)(2t-1)2\,dt\right)^\frac{q}{p}d\theta,\quad f\in\H(\D),
	\end{split}
	\end{equation*}
where the last weight satisfies
	$$
	\int_\rho^1\om(2t-1)(2t-1)2\,dt=\int_{2\rho-1}^1\om(r)r\,dr\lesssim\int_\rho^1\om(r)r\,dr,\quad \frac12\le\rho<1,
	$$
by the hypothesis $\om\in\DD$. Hence \eqref{trivial} and Theorem~\ref{th:radialmaximal} yield
	\begin{equation*}
	\begin{split}
	\|f\|_{AT^q_p(\om)}^q
	\lesssim\int_0^{2\pi}\left(\int_\frac12^1N_2(f)^p\left(re^{i\theta}\right)\om(r)r\,dr\right)^\frac{q}{p}d\theta
	\le\|N_2(f)\|_{L^q_p(\om)}^q
	\lesssim\|f\|_{AL^q_p(\om)}^q,\quad f\in\H(\D),
	\end{split}
	\end{equation*}
and thus $AL^q_p(\om)\subset AT^q_p(\om)$.

To see that $AT^q_p(\om)\subset AL^q_p(\om)$, write
	$$
	\Phi_{f,M,p}(re^{i\theta})=\frac{1}{1-r}\int_{|t-\theta|<1-r}N_M(f)^p(re^{it})\,dt,\quad 0\le r<1,\quad 1\le M<\infty,
	$$
for short. For each $-M< K<M$ we have
	$$
	N_M(f)(re^{i(K(1-r)+\theta)})\le N_M(f)(se^{i(K(1-s)+\theta)}),\quad 0\le r\le s<1,\quad \theta\in\mathbb{R},
	$$
and hence
	\begin{equation}\label{Eq:increasing}
	\begin{split}
	\Phi_{f,M,p}(re^{i\theta})
	&\le\frac{1}{1-r}\int_{|t-\theta|<1-r}N_M(f)^p(se^{i((t-\theta)\frac{1-s}{1-r}+\theta)})\,dt\\
	&=\Phi_{f,M,p}(se^{i\theta}),\quad 0\le r\le s<1,\quad \theta\in\mathbb{R},
	\end{split}
	\end{equation}
that is, $r\mapsto\Phi_{f,M,p}(re^{i\theta})$ is non-decreasing on $[0,1)$ for each fixed $\theta\in\R$, $0<p<\infty$,  $1\le M<\infty$ and $f:\D\to\C$. Since $re^{i\theta}\in\Gamma\left(\frac{1+r}{2}e^{it}\right)$ whenever $|\theta-t|<1-\frac{1+r}{2}$, we have
	\begin{equation*}
	\begin{split}
	\|f\|_{AL^q_p(\om)}^q
	&\lesssim\int_0^{2\pi}\left(\int_0^1\Phi_{f,1,p}\left(\frac{1+r}{2}e^{i\theta}\right)\om(r)r\,dr\right)^\frac{q}{p}\,d\theta.
	\end{split}
	\end{equation*}
Now that $r\mapsto\Phi_{f,1,p}\left(\frac{1+r}{2}e^{i\theta}\right)$ is non-decreasing, we may proceed as in the proof of the inclusion $AL^q_p(\om)\subset AT^q_p(\om)$ to obtain
	\begin{equation*}
	\begin{split}
	\|f\|_{AL^q_p(\om)}^q
	&\lesssim\int_0^{2\pi}\left(\int_0^1\Phi_{f,1,p}\left(re^{i\theta}\right)\om(r)r\,dr\right)^\frac{q}{p}\,d\theta
	=\|N(f)\|_{T^q_p(\om)}.
	\end{split}
	\end{equation*}
Theorem~\ref{th:tentmaximal} now completes the proof of $AT^q_p(\om)\subset AL^q_p(\om)$.
\end{Prf}

\medskip

Our next goal is to use \eqref{AAA} and \eqref{BBB} to obtain another equivalent norms in $AL^q_p(\om)$ and $AT^q_p(\om)$, provided $\om\in\DDD$. To do this, we will need three basic lemmas on doubling weights. For each $1<K<\infty$ and a radial weight~$\omega$, consider the sequence defined by
	\begin{equation}\label{eq:rhon}
  \rho_n=\rho_{n}(\omega,K)=\min \lbrace 0\leq r <1 \, : \widehat{\omega}(r) = \widehat{\omega}(0)K^{-n}\rbrace,\quad n\in\N\cup\{0\}.
	\end{equation}
It is strictly increasing, $\rho_0=0$ and $\lim_{n\to\infty}\rho_{n}=1$. For a proof of the following result, see \cite[Lemma~2.1]{PelSum14}.

\begin{lemma}\label{Lem_D_hat}
Let $\omega$ be a radial weight. Then the following statements are equivalent:
\begin{enumerate}
    \item[\rm(i)] $\omega \in \widehat{\mathcal{D}}$;
    \item[\rm(ii)] There exist $K= K(\omega)>1$ and $C=C(\omega, K)>1$ such that
    \[
    1-\rho_n(\omega, K) \geq C (1- \rho_{n+1}(\omega, K)),\quad n \in \mathbb{N} \cup \lbrace 0 \rbrace;
    \] 
    \item[\rm(iii)] There exist $C=C(\omega)\geq 1$ and $\beta_{0}=\beta_{0}(\omega)>0$ such that
    \[
    \widehat{\omega}(r)\leq C\left( \frac{1-r}{1-t}\right)^{\beta}\widehat{\omega}(t), \hspace{4mm} 0\leq r \leq t <1,
    \]
    for all $\beta\ge\beta_{0}$;
     \item[\rm(iv)] There exist $C=C(\omega)\geq 1$ and $\eta_{0}=\eta_{0}(\omega)>0$ such that
    \[
   \int_{\D} \frac{\omega(\zeta)}{|1-z\zeta|^\eta}\,dA(\zeta)\le C\frac{\omg(z)}{(1-|z|)^{\eta-1}},\quad z\in\D,
    \]
    for all $\eta\ge\eta_{0}$.
\end{enumerate}
\end{lemma}

The next result is a counterpart of Lemma~\ref{Lem_D_hat} for the class $\Dd$, and it is obtained by methods similar to those yielding the said lemma.

\begin{lemma}\label{Lem_D_check}
Let $\omega$ be a radial weight. Then the following statements are equivalent:
    \begin{enumerate}
    \item[\rm(i)] $\omega\in\widecheck{\mathcal{D}}$;
    \item[\rm(ii)] For some (equivalently for each) $K>1$, there exist $C=C(\omega,K)>0$ such that
    \[
    1-\rho_n(\omega,K)\le C(1-\rho_{n+1}(\omega,K)),\quad n \in \mathbb{N} \cup \lbrace 0 \rbrace;
    \]
    \item[\rm(iii)] There exist $C=C(\omega)>0$ and $\alpha_{0}=\alpha_{0}(\omega)>0$ such that
    \[
    \widehat{\omega}(s)\le C\left( \frac{1-s}{1-t}\right)^{\alpha}\widehat{\omega}(t), \quad 0\le t\le s<1,
    \]
for all $0<\alpha\le\alpha_{0}$.
    \end{enumerate}
\end{lemma}
		
The equivalence between (i) and (ii) in the next lemma is an immediate consequence of Lemmas~\ref{Lem_D_hat} and \ref{Lem_D_check}. The other equivalences follow from \cite[Theorems~8 and 9]{PR2020} and \cite[Lemma~9(vii)(viii)]{PGRV2024}.

\begin{lemma}\label{Lem_D}
    Let $\omega$ be a radial weight. Then the following statements are equivalent:
    \begin{enumerate}
        \item[\rm(i)] $\omega\in \mathcal{D}; $
        \item[\rm(ii)] There exist $K=K(\omega)>1$, $C_1=C_1(\omega,K)>1$ and  $C_2=C_2(\omega,K)\geq C_1$ such that
        \[
        C_1(1- \rho_{n+1}(\omega, K))\leq 1- \rho_n(\omega, K) \leq C_{2}(1- \rho_{n+1}(\omega, K)),\quad n \in \mathbb{N} \cup \lbrace 0 \rbrace;
        \]
        \item[\rm(iii)] $\widehat{\omega}\asymp \widehat{\widetilde{\omega}}$ on $[0,1)$;
        \item[\rm(iv)] $\widetilde{\omega}\in\mathcal{D}$.
    \end{enumerate}
\end{lemma}

\medskip

\noindent\emph{Proof of Theorem~\ref{th:samespace2}.} 
Assume (i), that is, $\om\in\DDD$. Then $\widehat{\widetilde{\om}}\asymp\widehat{\om}$ on $[0,1)$ by Lemma~\ref{Lem_D}(iii). Now that $r\mapsto\Phi_{f,M,p}(re^{i\theta})$ is non-decreasing by \eqref{Eq:increasing}, \eqref{trivial} and Theorem~\ref{th:tentmaximal} yield
	\begin{equation*}
	\begin{split}
	\|f\|_{AT^q_p(\widetilde\om)}
	\le\|\Phi_{f,M,p}\|^{\frac{1}{p}}_{L^{\frac{q}{p}}_1(\widetilde\om)}
	\lesssim\|\Phi_{f,M,p}\|^{\frac{1}{p}}_{L^{\frac{q}{p}}_1(\om)}
	=\|N_M(f)\|_{T^q_p(\om)}
	\lesssim\|f\|_{AT^q_p(\om)},\quad f\in\H(\D).
	\end{split}
	\end{equation*}
Since this argument readily gives the converse implication when the roles of $\om$ and $\widetilde{\om}$ are interchanged, we have $\|f\|_{AT^q_p(\widetilde\om)}\asymp\|f\|_{AT^q_p(\om)}$ for all $f\in\H(\D)$. Thus (ii) is satisfied. The proof that (i) also implies (iii) is similar to the argument above, and hence it is omitted.

Conversely, if (ii) or (iii) is satisfied, then by testing with the monomials $z\mapsto z^n$, one obtains the moment condition $\om_{x}=\int_0^1r^{x}\om(r)\,dr\asymp\widetilde\om_{x}$ for all $1\le x<\infty$. Standard arguments together with \cite[Theorems~1 and 3]{PGRV2024} and Lemma~\ref{Lem_D} yield $\om\in\DDD$.
\hfill$\Box$

\medskip 

Aiming to prove Theorem~\ref{th:counterexaple} some notation and previous results are needed.
A radial weight $\omega$ is rapidly decreasing if it satisfies the following conditions:
\begin{enumerate}
\item $\omega=e^{-\varphi}$, where $\varphi\in C^2(\D)$ is a radial function such that its Laplacian satisfies $\Delta\varphi\ge B_{\varphi}>0$ in $\D$ for some positive constant $B_{\varphi}$ depending only on $\varphi$; 
\item $\left(\Delta\varphi\right)^{-1/2}\asymp\tau$, where $\tau$ is a radial positive differentiable function that decreases to $0$, as
$|z|\rightarrow 1^{-}$, and $\lim_{r\to 1^-}\tau'(r)=0$;
\item There exists a constant $C>0$  such that either $\tau(r)(1-r)^{-C}$ is increasing for $r$ close to $1$, or
	$$
	\lim_{r\to 1^-}\tau'(r)\log\frac{1}{\tau(r)}=0.
	$$
\end{enumerate}
The class of rapidly decreasing weights is denoted by $\mathcal{W}$. This class does not include the standard weights, but it contains, for example, the exponential type weights
	\begin{displaymath}
	z\mapsto\exp\left(\frac{-c}{(1-|z|)^\alpha}\right),\quad z\in\D,
	\end{displaymath}
and the double exponential type weights
	\begin{displaymath}
	z\mapsto\exp\left(\exp\left(\frac{-c}{1-|z|}\right)\right),\quad z\in\D,
	\end{displaymath}
where $0<c,\alpha<\infty$ are fixed.

Let $B^\om_a$ denote the Bergman reproducing kernel of $A^2_\om$ associated to a point $a\in\D$. In the next result, we gather together some known facts on the Bergman spaces induced by weights in $\mathcal{W}$ that are useful for our purpose.

\begin{letterlemma}\label{th:W}
Let $\om\in\mathcal{W}$. Then the following statements hold:
\begin{enumerate}
\item[\rm(i)] There exists $\rho_0\in (0,1)$ such that $\| B^\om_a\|^2_{A^2_\omega}\asymp\tau(a)^{-2}\om(a)^{-1}$ for all $\rho_0\le |a|<1$;
\item[\rm(ii)] There exists $\delta>0$ such that
	$$
	| B^\om_a(z)|\asymp \| B^\om_a\|_{A^2_\omega}\| B^\om_z\|_{A^2_\omega}, \quad z\in D(a,\delta\tau(a)),\quad a\in\D;
	$$
\item[\rm(iii)] There exists $\delta>0$ such that $\tau(z)\asymp\tau(a)$ for all $z\in D(a,\delta\tau(a))$ and $a\in\D$;
\item[\rm(iv)] $\lim_{r\to 1^{-}}\frac{1-r}{\tau(r)}=\infty$.
\end{enumerate}
\end{letterlemma}

\begin{proof}
Parts (i) and (iii) are proved in \cite[Corollary~1]{Pau-Pelaez:JFA2010} and \cite[Lemma~2.1]{Pau-Pelaez:JFA2010}, respectively. Next, 
bearing in mind the hypothesis $\lim_{r\to 1^-}\tau'(r)=0$, it follows that $\varphi=\log\frac{1}{\om}\in\mathcal{W}_0$, where $\mathcal{W}_0$ is the class of functions considered in \cite{ATTY,HLSJFA}. Hence Part (ii) follows from \cite[(18)]{HLSJFA} and Part(i), see also \cite[Lemma~E]{Atesis}. Finally, (iv) follows from the fact that $\lim_{r\to 1^-}\tau'(r)=0$.
\end{proof}

\begin{lemma}\label{le:Walpha}
Let $\mu\in\mathcal{W}$ and $\a>0$. Then there exists $C=C(\alpha,\mu)>0$ such that 
	$$
	\om(r)=e^{-Cr^2}\frac{\mu(r)}{(1-r^2)^\alpha} \in \mathcal{W}.
	$$
\end{lemma}

\begin{proof}
By \cite[Lemma~2.3]{Pau-Pelaez:JFA2010} and Lemma~\ref{th:W}(iv), $\lim_{r \to 1^-}\om(r)=0$. In particular, $\om$ is a radial weight. Moreover, $\om=e^{-\Psi}$ with 
$$ \Psi(r) =Cr^2+\varphi(r)+\alpha\log(1-r^2), \quad 0\le r<1,$$
where     $\mu=e^{-\varphi}$. Therefore
$$\Delta \Psi(r) = 4C+ \Delta \varphi(r)- \alpha\left( \frac{2}{1-r^2}+ \frac{2(1+r^2)}{(1-r^2)^2} \right), \quad 0<r<1.$$
Then,  by using   Lemma~\ref{th:W}(iv) again we deduce
$$\lim_{r\to 1^-} \frac{\Delta \Psi(r) }{\Delta \varphi(r) }=1.$$
Consequently, $C$ can be chosen large enough so that $\inf_{0<r<1}\Delta \Psi(r)>0$. Moreover, the  differentiable function $\tau$ such that  $\left(\Delta\varphi\right)^{-1/2}\asymp\tau$ 
also satisfies $\left(\Delta\Psi\right)^{-1/2}\asymp\tau$. Thus $\om\in \mathcal{W}$.
\end{proof}

\begin{Prf}{\em{Theorem~\ref{th:counterexaple}.}}
Let $p=1$ and  $q=2$. By Lemma~\ref{le:Walpha}, there exists $C=C(\mu)>0$ such that  $\om(r)=e^{-Cr^2}\frac{\mu(r)}{(1-r^2)^{1/2}} \in \mathcal{W}.$ In order to prove that $AT_1^2(\mu)$ is not embedded into $AL^2_1(\mu)$ it is enough to show that $I:A^1_\om\to AT_1^2(\mu)$ is bounded but $I:A^1_\om\to AL_1^2(\mu)$ is not. Let us write $\om=e^{-\varphi}$ and let $\tau$ be a
differentiable function  such that  $\left(\Delta\varphi\right)^{-1/2}\asymp\tau$.
 Bearing in mind that $\lim_{r\to 1^-}\tau'(r)=0$ by the hypothesis, it follows that $\varphi=\log\frac{1}{\om}\in\mathcal{W}_0$, where $\mathcal{W}_0$ is the class of functions considered in \cite{ATTY,HLSJFA}. Consequently, by \cite[Theorem~3.3]{ATTY}, $I: A^1_\om\to  AT_1^2(\mu)$ is bounded because for each $\delta>0$ small enough
	$$
	\sup_{a\in\D}\frac{\int_{D(a,\delta\tau(a))} \om^{-1} \mu \,dA}{(1-|a|)^{\frac12}\tau^2(a)}
	\asymp\sup_{a\in\D}\frac{\int_{D(a,\delta\tau(a))}(1-|z|)^{\frac12} \,dA(z)}{(1-|a|)^{\frac12}\tau^2(a)}\asymp 1.
	$$
On the other hand, for some $\delta>0$ small enough such that Lemma~\ref{th:W}(ii)-(iii) holds, take $\delta_1>0$ and intervals
$J_a\subset \T$,  $I_a=[a-\delta_1\tau(a), a+\delta_1\tau(a)]  \subset [0,1)$ with  $|I_a|=|J_a|$ and 
$I_a\times J_a\subset D(a,\delta\tau(a))$. Next, consider the family of analytic functions $f_a(z)=\left( \frac{B^\omega_a(z)}{\| B^\om_a\|_{A^2_\omega}}\right)^2$, $\, a \in\D$. Then Lemma~\ref{th:W}(i)--(iii) yields  
	\begin{equation*}
	\begin{split}
	\|f_a\|^2_{ AL_1^2(\mu)}
	&\ge\int_{J_a}\left(\int_{I_a}|f_a(re^{it})|\mu(r)r\,dr\right)^2\,dt
	=\frac{\int_{J_a}\left(\int_{I_a}|B^\omega_a(re^{it})|^2\mu(r)r\,dr\right)^2\,dt}{\| B^\om_a\|_{A^2_\omega}^4}\\
	&\asymp\frac{\int_{J_a}\left(\int_{I_a}\left(\|B^\om_a\|_{A^2_\omega}\|B^\om_{re^{it}}\|_{A^2_\omega}\right)^2\mu(r)r\,dr\right)^2\,dt}{\| B^\om_a\|_{A^2_\omega}^4}
	=\int_{J_a}\left( \int_{I_a}\| B^\om_{re^{it}}\|_{A^2_\omega}^2\mu(r)r\,dr\right)^2\,dt\\
	&\asymp\tau(a)\left( \int_{I_{a}}(1-r)^{\frac12}\tau(r)^{-2}  \,dr\right)^2
 \asymp \frac{1-|a|}{\tau(a)},\quad \rho_0\le|a|<1.
\end{split}
\end{equation*}
Therefore Lemma~\ref{th:W}(iii) implies  $\sup_{a\in\D}\|f_a\|_{ AL_1^2(\mu)}=\infty$, which together with the fact $\sup_{a\in\D}\|f_a\|_{A^1_\om}=1$ shows that $I:A^1_\om\to AL^2_1(\mu)$ is not bounded. This finishes the proof.
\end{Prf}
\medskip

We will next prove Proposition~\ref{th:intro1}. To do this, we recall that the Hardy-Littlewood maximal function of $g\in L_{\textrm{loc}}^1(\T)$ is defined by
	\begin{equation}\label{eq:HLmaximal}
	M(g)(\xi) = \sup_{\xi\in I} \frac{1}{|I|}\int_{I} |g(z)| \, dz,\quad \xi\in\T,
	\end{equation}  
where the supremum is taken over all arcs $I$ on $\T$ containing the point $\xi$. 

\medskip

\noindent\emph{Proof of Proposition~\ref{th:intro1}.} To prove (i), let $p<q$. Then $\frac{q}{p}>1$, and hence
    \begin{align*}
		\|f\|^p_{T^q_p(\omega)}
		=\sup_{g\in B}\int_{\T}
		\left(\int_{\Gamma(\xi)}|f(z)|^{p}\omega(z)\frac{dA(z)}{1-|z|}\right)|g(\xi)|\,|d\xi|,
    \end{align*}
where we have written $B=B_{L^{(\frac{q}{p})'}(\T)}$ for the closed unit ball of $L^{(\frac{q}{p})'}(\T)$. By denoting $I(z)=\left\{\xi\in\T:z\in \Gamma(\xi)\right\}$, Fubini's theorem now yields
		\begin{align*}
    \|f\|^p_{T^q_p(\omega)}
		&=\sup_{g\in B}\int_{\D}|f(z)|^{p}\omega(z)\left(\int_{I(z)}|g(\xi)|\frac{|d\xi|}{1-|z|}\right)\,dA(z)\\
    &\le\sup_{g\in B}\int_{0}^{2\pi}\left(\int_{0}^{1} |f(re^{i\theta})|^{p}\omega(r)r\,dr\right) 
		\left(\sup_{r\in [0,1)}\int_{I(re^{i\theta})} |g(\xi)|\ \frac{|d\xi|}{1-r}\right)\ \frac{d\theta}{\pi}.
		\end{align*}
The length of $I(z)$ satisfies
	\begin{equation}\label{vjhblklkj}
	|I(z)|=\int_{\T} \chi_{\Gamma(\xi)}(z)\ |d\xi|\asymp (1-|z|),\quad z\in\D.
	\end{equation}
Therefore, by Hölder's inequality and the Hardy-Littlewood Maximal theorem, we get
	\begin{align*}
	\|f\|^p_{T^q_p(\omega)}
	&\lesssim\sup_{g\in B}\int_{0}^{2\pi}\left(\int_{0}^{1}|f(re^{i\theta})|^{p}\omega(r)r\,dr\right) M(g)(e^{i\theta})d\theta\\
	&\lesssim\|f\|^p_{L^q_p(\om)}\sup_{g\in B}\|M(g)\|_{L^{(\frac{q}{p})'}(\T)}
	\lesssim\|f\|^p_{L^q_p(\om)}\sup_{g\in B}\|g\|_{L^{(\frac{q}{p})'}(\T)}
	\le\|f\|^p_{L^q_p(\om)},
	\end{align*}
that is, $L^q_p(\omega)\subset T^q_p(\omega)$. 

To see that the inclusion is strict, consider the function 
    \[
    f(re^{i\theta})
		=\left\{\begin{array}{cc}
      \frac{(1-r)^{\frac{1}{p}}}{|\theta|^{\frac{1}{q}}\omega^{\frac{1}{p}}(r)},&\quad\omega(r)\theta\neq 0,  \\
        0,&\quad \omega(r)\theta = 0, 
    \end{array}\right.
    \]
and write $\chi_\om$ for the characteristic function of the set of positive measure on which $\om$ is different from zero. Then
    \begin{align*}
		\|f\|_{L^q_p(\om)}^{q}
		%&=\frac1{2\pi}\int_{0}^{2\pi}\left(\int_{0}^{1}|f(re^{i\theta})|^{p}\omega(r)r\,dr\right)^{\frac{q}{p}}\,d\theta\\
		&=\frac1{2\pi}\int_{0}^{2\pi}\frac{d\theta}{|\theta|}\left(\int_{0}^{1}\chi_\om(r)(1-r)r\,dr\right)^{\frac{q}{p}}=\infty
    \end{align*}
and 
	\begin{align*}
	\|f\|_{T_p^q(\omega)}^{q}	%&=\int_{0}^{2\pi}\left(\int_{0}^{1}\int_{|\theta|<1-r}|f(re^{i(\theta+t)})|^{p}\,d\theta\frac{\omega(r)r}{1-r}\frac{dr}{\pi}\right)^{\frac{q}{p}}\,dt\\
  &\le\int_{0}^{2\pi}\left(\int_{0}^{1}\left(\int_{|\theta|<1-r}\frac{d\theta}{|\theta+t|^{\frac{p}{q}}}\right)\,dr \right)^{\frac{q}{p}}\,dt
  \le\int_{0}^{2\pi} \left(\int_{-1}^{1}\frac{d\theta}{|\theta+t|^{\frac{p}{q}}}\right)^{\frac{q}{p}}\,dt
	<\infty
\end{align*}
because $\frac{p}{q}<1$. Thus $f\in T_p^q(\omega)\setminus L_p^q(\omega)$.

The statement (ii) is a well-known consequence of Fubini's theorem and \eqref{vjhblklkj}.

To prove (iii), let $q<p$, and choose $b>\frac{1}{q}$. By \cite[Theorem~1, p.~303] {Benedek_Panzone}, we have
	\begin{align*}
	\|f\|_{L^{q}_{p}(\om)}&=\||f|^{\frac1b}\|_{L^{bq}_{bp}(\om)}^{b}
	=\sup_{g\in B}  \left(\int_{\D} |f(z)|^{\frac1b}|g(z)|\omega(z)\,dA(z)\right)^b,
	\end{align*}
where we have  written $B=B_{L^{(bq)'}_{(bp)'}(\om)}$ for short. By applying \eqref{vjhblklkj}, Fubini's theorem and H\"older's inequality twice, we deduce
	\begin{align*}
	\|f\|_{L^q_p(\om)}^{\frac1b}
	&\asymp\sup_{g\in B}\int_{\D}|f(z)|^{\frac1b}|g(z)|\omega(z)\left(\int_{\T}\chi_{\Gamma(\xi)}(z)\,|d\xi|\right)\,\frac{dA(z)}{1-|z|}\\
	&=\sup_{g\in B}\int_{\T}\left(\int_{\Gamma(\xi)}|f(z)|^{\frac1b}|g(z)|\omega(z)\,\frac{dA(z)}{1-|z|}\right)\,|d\xi|\\ 
	&\le\sup_{g\in B}\||f|^{\frac1b}\|_{T_{bp}^{bq}(\omega)}\|g\|_{T^{(bq)'}_{(bp)'}(\om)}
	=\|f\|_{T^{q}_{p}(\om)}^\frac1b\sup_{g\in B}\|g\|_{T^{(bq)'}_{(bp)'}(\om)}.
	\end{align*}
But $bq<bp$ by the hypothesis, and hence $(bq)'>(bp)'$. Therefore Part (i) implies
	\begin{align*}
	\sup_{g\in B}\|g\|_{T^{(bq)'}_{(bp)'}(\om)}\lesssim\sup_{g\in B}\|g\|_{L^{(bq)'}_{(bp)'}(\om)}=1.
	\end{align*}
Thus we have shown that $T^q_p(\omega)\subset L^q_p(\omega)$. 

To see that the inclusion is strict, consider the function 
   \[
   f(re^{i\theta})=\left\{\begin{array}{cc}
      \frac{(1-r)^{\frac{1}{p}}}{|\theta|^{\frac{1}{p}}\omega^{\frac{1}{p}}(r)},&\quad\omega(r)\theta\neq 0,  \\
       0, &  \quad \omega(r)\theta = 0. 
    \end{array}\right.
    \]
Then
	\begin{align*}
	\|f\|_{L^q_p(\om)}^{q}
	%&=\frac1{2\pi}\int_{0}^{2\pi}\left(\int_{0}^{1}|f(re^{i\theta})|^{p}\omega(r)r\,dr\right)^{\frac{q}{p}}\,d\theta\\
	&\le\frac1{2\pi}\int_{0}^{2\pi}\frac{d\theta}{|\theta|^{\frac{q}{p}}}\left(\int_{0}^{1}(1-r)\,dr\right)^{\frac{q}{p}}
	<\infty
  \end{align*}
because $\frac{q}{p}<1$. Since $\om$ is a weight, there exists a constant $r_0=r_0(\om)\in(0,1)$ such that $\int_0^{r_0}\chi_\om(r)r\,dr>0$. This observation together with Fubini's theorem yields
	\begin{align*}
	\|f\|_{T_p^q(\omega)}^{q}
  &\ge\frac1{2\pi}\int_{0}^{2\pi}\left(\int_{0}^{1}\chi_\om(r)r
	\left(\int_{r-1}^0\frac{d\theta}{|\theta+t|}\right)dr\right)^{\frac{q}{p}} dt\\
	&=\frac1{2\pi}\int_{0}^{2\pi}\left(\int_{-1}^{0}\frac{1}{|\theta+t|}
	\left(\int_0^{1+\theta}\chi_\om(r)r\,dr\right)d\theta\right)^{\frac{q}{p}}dt\\
	&\ge\frac1{2\pi}\int_{0}^{2\pi}\left(\int_{r_0-1}^{0}\frac{1}{|\theta+t|}
	\left(\int_0^{r_0}\chi_\om(r)r\,dr\right)d\theta\right)^{\frac{q}{p}}dt\\
	&\gtrsim\int_{0}^{2\pi}\left(\int_{r_0-1}^{0}\frac{d\theta}{|\theta+t|}\right)^{\frac{q}{p}}dt
	\ge\int_{0}^{\frac{1-r_0}{2}}\left(\int_{r_0-1+t}^{t}\frac{dx}{|x|}\right)^{\frac{q}{p}}dt\\
	&\ge\frac{1-r_0}{2}\left(\int_0^{\frac{1-r_0}{2}}\frac{dx}{x}\right)^\frac{q}{p}=\infty,
	\end{align*}
and thus $f\in L_p^q(\omega)\setminus T_p^q(\omega)$.\hfill$\Box$

\section{Littlewood-Paley inequalities}\label{S3}

We begin with an auxiliary result which guarantees that, for each radial weight $\omega$, the norm convergence in either $AT^{q}_{p}(\omega)$ or $AL^{q}_{p}(\omega)$ implies the uniform convergence on compact subsets. As usual, we write
	\begin{equation*}
	\begin{split}
	M_p(r,f)=\left(\frac{1}{2\pi}\int_0^{2\pi}|f(re^{i\theta})|^p\,d\theta\right)^\frac1p, \quad 0<r<1,
	\end{split}
	\end{equation*}
for the $L^p$-mean of the restriction of $f$ to the circle of radius $r$, and $M_{\infty}(r,f)=\max_{|z|=r}|f(z)|$ for the maximum modulus.

\begin{lemma}\label{Prop2.3}
Let $0<p,q<\infty$ and $n\in\N\cup\{0\}$, and let $\omega$ be a radial weight. Then there exist constants $C_1=C_1(p,q,n,\om)>0$ and $C_2=C_2\left(p,q,\frac{1-\rho}{\rho-r}\right)>0$ such that
		\[
    M_\infty(r,f^{(n)})
		\le C_1\frac{\min\left\{\|f\|_{AL^{q}_{p}(\omega)},C_2\|f\|_{AT^{q}_{p}(\omega)}\right\}}{\widehat{\omega}(\rho)^{\frac{1}{p}}(\rho-r)^{\frac{1}{q}+n}},
		\quad 0\le r<\rho<1,\quad f\in \H(\D).
    \]
In particular, if $\omega\in\DD$, then the choice $\rho=\frac{1+r}{2}$ gives
    \[
    M_\infty(r,f^{(n)})
		\lesssim\frac{\min\left\{\|f\|_{AL^{q}_{p}(\omega)},\|f\|_{AT^{q}_{p}(\omega)}\right\}}{\widehat{\omega}(r)^{\frac{1}{p}}(1-r)^{\frac{1}{q}+n}},
		\quad 0\le r<1,\quad f\in \H(\D).
    \]
\end{lemma}

\begin{proof}
It is well known that $|f(\zeta)|^q\le\|f\|^q_{H^q}/(1-|\zeta|)$ for all $\zeta\in\D$, $0<q<\infty$ and $f\in\H(\D)$, see \cite[Theorem~9.1]{Zhu2007} for details. An application of this inequality to $\zeta\mapsto f\left(\rho\zeta\right)$ at $\zeta=z/\rho$ yields 
	\begin{align*}
	M_\infty^q(r,f)
	\le\rho\frac{M_q^q\left(\rho,f\right)}{\rho-r},\quad 0<r<\rho<1,\quad f\in\H(\D).
	\end{align*}
This estimate together with Theorem~\ref{th:radialmaximal} gives
	\begin{equation*}
	\begin{split}
	\|f\|_{AL^{q}_{p}(\omega)}^q
	&\asymp\frac1{2\pi}\int_0^{2\pi}\left(\int_0^1R(f)^p(te^{i\theta})\om(t)\,dt\right)^\frac{q}{p}d\theta
	\ge\widehat\om\left(\rho\right)^\frac{q}{p}\frac1{2\pi}\int_0^{2\pi}R(f)^q\left(\rho e^{i\theta}\right)\,d\theta\\
	&\ge\widehat\om\left(\rho\right)^\frac{q}{p}M_q^q\left(\rho,f\right)
	\ge\widehat\om\left(\rho\right)^\frac{q}{p}(\rho-r)M^q_\infty(r,f),\quad 0<r<\rho<1,\quad f\in\H(\D),
	\end{split}
	\end{equation*}
and thus the statement for $AL^q_p(\om)$ is proved in the case $n=0$. The general case $n\in\N$ is a consequence of $n$ applications of the inequality
	$$
	M_\infty(r,f')\lesssim\frac{M_\infty(\rho,f)}{\rho-r},\quad 0\le r<\rho<1,\quad f\in\H(\D),
	$$
which is a consequence of the generalized Cauchy integral formula, and the case $n=0$ just proved. Details are omitted.

To deal with $AT^q_p(\om)$, let $M=M(r,\rho)= \max\left\{1,2\frac{1-\rho}{\rho-r}\right\}$. Since $\frac{r+\rho}{2}e^{i\theta}\in\Gamma_M\left(\rho e^{it}\right)$ whenever $|\theta-t|<1-\rho$, we have 
	$$
	\Phi_{f,M,p}(\rho e^{i\theta})=\frac{1}{1-\rho}\int_{|t-\theta|<1-\rho}N_M(f)^p(\rho e^{it})\,dt\ge\left|f\left(\frac{r+\rho}{2}e^{i\theta}\right)\right|^p.
	$$
Hence Theorem~\ref{th:tentmaximal} and \eqref{Eq:increasing} yield
	\begin{equation*}
	\begin{split}
	\|f\|_{AT^{q}_{p}(\omega)}^q
	&\asymp\frac1{2\pi}\int_0^{2\pi}\left(\int_0^1\Phi_{f,M,p}(te^{i\theta})\om(t)\,dt\right)^\frac{q}{p}d\theta
	\ge\widehat\om\left(\rho\right)^\frac{q}{p}\frac1{2\pi}\int_0^{2\pi}\Phi_{f,M,p}^\frac{q}{p}\left(\rho e^{i\theta}\right)\,d\theta\\
	&\ge\widehat\om\left(\rho\right)^\frac{q}{p}M_q^q\left(\frac{r+\rho}{2},f\right)
	\ge\widehat\om\left(\rho\right)^\frac{q}{p}\frac{\rho-r}{2}M^q_\infty(r,f),\quad 0<r<\rho<1,\quad f\in\H(\D),
	\end{split}
	\end{equation*}
and thus the case $n=0$ of the statement for $AT^q_p(\om)$ is proved. For $n\in\N$ the assertion follows as in the case of $AL^q_p(\om)$.
\end{proof}

The pseudohyperbolic disc centered at $z\in\D$ and of radius $0<r<1$ is the set $\Delta(z,r)=\{\zeta\in\D:|\varphi_z(\zeta)|<r\}$, where $\varphi_z(\zeta)=(z-\zeta)/(1-\overline{z}\zeta)$ for all $z,\zeta\in\D$. It coincides with the Euclidean disc $D(Z,R)$, where 
	$$
	Z=\frac{1-r^2}{1-|z|^2r^2}z\quad\textrm{and}\quad R=\frac{1-|z|^2}{1-|z|^2r^2}r.
	$$

\begin{theorem}\label{theo_LP_1} Let $0<p,q<\infty$ and $n\in\N$, and let $\omega$ be a radial weight. Then $\om\in\widehat{\mathcal{D}}$ if and only if there exists a constant $C=C(p,q,n,\om)>0$ such that
\[
	\|f^{(n)}(1-|\cdot|)^{n}\|_{L_{p}^{q}(\omega)}+\sum_{j=0}^{n-1}|f^{(j)}(0)|
	\le C\|f\|_{AL_{p}^{q}(\omega)},\quad f\in\H(\D).
\]
Similarly, $\om\in\widehat{\mathcal{D}}$ if and only if there exists a constant $C=C(p,q,n,\om)>0$ such that
\[
	\|f^{(n)}(1-|\cdot|)^{n}\|_{T_{p}^{q}(\omega)}+\sum_{j=0}^{n-1}|f^{(j)}(0)|
	\le C\|f\|_{AT_{p}^{q}(\omega)},\quad f\in\H(\D).
\]
\end{theorem}

\begin{proof}
Fix $0<r<t<1$. A change of variable and an application of 
	\begin{align*}
	M_p(r,f')
	\lesssim\frac{M_p\left(\rho,f\right)}{\rho-r},\quad 0\le r<\rho<1,\quad f\in\H(\D),
	\end{align*}
the proof of which can be found in \cite{Duren1970}, yield
	\begin{equation*}
	\begin{split}
	\int_{\Delta(z,r)}|f'(\z)|^p\,dA(\z)
	%&\asymp(1-|z|)^{2-p}\int_0^1 M_p^p((f\circ\varphi_z)',rs)\,ds\\
	%&\lesssim(1-|z|)^{2-p}\int_0^1\frac{M_p^p\left(f\circ\varphi_z,\frac{rs+\e}{1+\e}\right)}{(1-rs)^p}\,ds\\
	%&=(1-|z|)^{2-p}\int_{\frac{\e}{1+\e}}^\frac{r+\e}{1+\e}\frac{M_p^p\left(f\circ\varphi_z,x\right)}{(1+\e)^p(1-x)^p}\frac{1+\e}r\,dx\\
	%&\lesssim(1-|z|)^{2-p}\frac{1+\e}{(1-r)^pr}\int_{\frac{\e}{1+\e}}^\frac{r+\e}{1+\e}M_p^p\left(f\circ\varphi_z,x\right)\,dx\\
	%&\lesssim(1-|z|)^{2-p}\int_0^tM_p^p\left(f\circ\varphi_z,x\right)x\,dx\\
	%&\asymp(1-|z|)^{-p}\int_{D\left(0,t\right)}|f(\varphi_z(\zeta))|^p|\varphi_z'(\zeta)|^2\,dA(\zeta)\\
	&\lesssim\frac1{(1-|z|)^{p}}\int_{\Delta(z,t)}|f(\z)|^p\,dA(\z),\quad z\in\D.
	\end{split}
	\end{equation*}
By using this estimate for $0<t<1$ sufficiently small, the subharmonicity of $|f'|^p$, \eqref{trivial}, Lemma~\ref{Prop2.3} and Theorem~\ref{th:radialmaximal} imply 
	\begin{equation*}
	\begin{split}
	\|f'(1-|\cdot|)\|_{L_{p}^{q}(\omega)}^q
	&\lesssim\int_0^{2\pi}\left(\int_0^1\left(\frac{1}{(1-s)^2}\int_{\Delta(se^{i\theta},r)}|f'(\zeta)|^p\,dA(\zeta)\right)(1-s)^p\om(s)s\,ds\right)^\frac{q}{p}d\theta\\
	&\lesssim\int_0^{2\pi}\left(\int_0^1\left(\frac{1}{(1-s)^2}\int_{\Delta(se^{i\theta},t)}|f(\zeta)|^p\,dA(\zeta)\right)\om(s)s\,ds\right)^\frac{q}{p}d\theta\\
	&\lesssim M^q_\infty\left(\frac12,f\right)
	+\int_0^{2\pi}\left(\int_0^1N(f)^p\left(\frac{1+s}{2}e^{i\theta}\right)\om(s)s\,ds\right)^\frac{q}{p}d\theta\\
	&\lesssim M^q_\infty\left(\frac12,f\right)
	+\|N(f)\|_{L^q_p(\om)}^q\lesssim\|f\|_{AL^q_p(\om)}^q.
	\end{split}
	\end{equation*}
A slight modification in this argument shows that $\|f'(1-|\cdot|)\|_{T_{p}^{q}(\omega)}\lesssim\|f\|_{AT^q_p(\om)}$, provided $\om\in\DD$. The general case $n\in\N$ readily follows from the argument above.

Conversely, by testing one of the inequalities with monomials easily give $\om\in\DD$, see the proof of \cite[Theorem~6]{PR19} for a similar argument.
\end{proof}

\begin{theorem}\label{Thm:LP} Let $0<p,q<\infty$, $n\in\N$ and $\omega\in\DDD$. Then there exists a constant $C=C(p,q,n,\om)>0$ such that
    \[
		\|f\|_{AT_{p}^{q}(\omega)}
		\le C\left(\|f^{(n)}(1-|\cdot|)^{n}\|_{T_{p}^{q}(\omega)}+\sum_{j=0}^{n-1}|f^{(j)}(0)|\right),\quad f\in\H(\D).
    \]
\end{theorem}

\begin{proof}
By  Theorem~\ref{th:samespace2} it is enough to prove the statement for  the average radial integrability spaces.
We may assume without loss of generality that $f(0)=0$. Let first $1<p<\infty$ and take $h(t)=(1-t)^{1-\frac{1-\e}{p}}$, where $0<\e<1$ will be fixed later. Then the estimate $|f(re^{i\theta})|\le\int_0^{r}|f'(te^{i\theta})|\,dt$, 
H\"older's inequality and Fubini's theorem yield
	\begin{equation*}
	\begin{split}
	\int_0^1|f(re^{i\theta})|^p\om(r)r\,dr
	&\le\int_0^1\left(\int_0^r|f'(te^{i\theta})|h(t)\frac{dt}{h(t)}\right)^p\om(r)r\,dr\\
	&\le\int_0^1\left(\int_0^r|f'(te^{i\theta})|^ph(t)^p\,dt\right)
	\left(\int_0^r\frac{dt}{h(t)^{p'}}\right)^{p-1}\om(r)r\,dr\\
	&\lesssim\int_0^1\left(\int_0^r|f'(te^{i\theta})|^p(1-t)^{p-1+\e}\,dt\right)
	\om_{[-\e]}(r)r\,dr\\
	&=\int_0^1|f'(te^{i\theta})|^p(1-t)^{p-1+\e}\left(\int_t^1\om_{[-\e]}(r)r\,dr\right)dt.
	\end{split}
	\end{equation*}
By \cite[Lemma~2]{PR2023} we may fix $\e=\e(\om)\in(0,1)$ sufficiently small such that $\int_t^1\om_{[-\e]}(r)\,dr\lesssim\widehat{\om}_{[-\e]}(t)$ for all $0\le t<1$, and thus
	\begin{equation*}
	\begin{split}
	\int_0^1|f(re^{i\theta})|^p\om(r)r\,dr
	&\lesssim\int_0^1N(f')(te^{i\theta})^p\widetilde\om_{[p]}(t)\,dt.
	\end{split}
	\end{equation*}
Further, as $1<p<\infty$ and $\om\in\Dd$ by the hypotheses, \cite[Lemma~9(v)]{PGRV2024} gives
	\begin{equation}\label{lejherlkhg}
	\int_r^1\widetilde\om_{[p]}(t)\,dt
	\le(1-r)^p\widehat{\om}(r)\lesssim\int_r^1\om_{[p]}(t)\,dt,\quad 0\le r<1.
	\end{equation}
Therefore \eqref{trivial} yields
	\begin{equation*}
	\begin{split}
	\int_0^1N(f')(te^{i\theta})^p\widetilde\om_{[p]}(t)\,dt
	\lesssim\int_0^1N(f')(te^{i\theta})^p\om_{[p]}(t)\,dt.
	\end{split}
	\end{equation*}
These estimates together with Theorem~\ref{th:radialmaximal} imply 
		\[
		\|f\|_{AL_{p}^{q}(\omega)}
		\lesssim\|N(f')(1-|\cdot|)\|_{L_{p}^{q}(\omega)}
		\lesssim\|f'(1-|\cdot|)\|_{L_{p}^{q}(\omega)},\quad f\in\H(\D),\quad f(0)=0.
    \]
The case $n=1$ for $AL_{p}^{q}(\omega)$ is now proved. Since $\om_{[np]}\in\Dd$ by \cite[Lemma~9(v)]{PGRV2024}, the general case follows from the estimate above.

Let now $0<p\le1$, and for each $0<r<1$ set $r_n=r_n(r)=\max\{1-2^{n}(1-r),0\}$. Further, choose $0<s<\rho<1$ such that $\Delta(te^{i\theta},s)\subset\Delta(r_ne^{i\theta},\rho)$ for all $r_{n+1}\le t\le r_n$ and $n\in\N\cup\{0\}$. Then the subharmonicity of $|f'|^p$ yields
	\begin{equation*}
	\begin{split}
	|f(re^{i\theta})|^p
	&\le\left(\sum_{n=0}^\infty\int_{r_{n+1}}^{r_n}|f'(te^{i\theta})|dt\right)^p
	\le\left(\sum_{n=0}^\infty\sup_{r_{n+1}\le t\le r_n}|f'(te^{i\theta})|(r_n-r_{n+1})\right)^p\\
	&\le\sum_{n=0}^\infty\left(\sup_{r_{n+1}\le t\le r_n}|f'(te^{i\theta})|^p\right)(1-r_n)^p
	\lesssim\sum_{n=0}^\infty\int_{\Delta(r_ne^{i\theta},\rho)}|f'(z)|^p\,dA(z)(1-r_n)^{p-2}\\
	&\asymp\sum_{n=0}^\infty\int_{\Delta(r_ne^{i\theta},\rho)}|f'(z)|^p(1-|z|)^{p-2}\,dA(z)\\
	&\lesssim\int_{\Gamma_M(\frac{r+K}{K+1}e^{i\theta})\cup D(0,\frac{K}{K+1})}|f'(z)|^p(1-|z|)^{p-2}\,dA(z)
	\end{split}
	\end{equation*}
for some sufficiently large $1<K,M<\infty$. By using the hypothesis $\om\in\DD$, it follows that
	\begin{equation*}
	\begin{split}
	\int_0^1|f(re^{i\theta})|^p\om(r)r\,dr
	&\lesssim\int_0^1\left(\int_{\Gamma_M(\frac{r+K}{K+1}e^{i\theta})\cup D\left( 0,\frac{K}{K+1}\right)}|f'(z)|^p(1-|z|)^{p-2}\,dA(z)\right)\om(r)r\,dr\\
	&\lesssim M^p_\infty\left(\frac{K}{K+1},f'\right)+\int_{\Gamma_M(e^{i\theta})\setminus D\left( 0,\frac{K}{K+1}\right)}|f'(z)|^p(1-|z|)^{p-2}\widehat{\om}\left((K+1)|z|-K\right)\,dA(z)\\
	&\lesssim M^p_\infty\left(\frac{K}{K+1},f'\right)+\int_{\Gamma_M(e^{i\theta})}|f'(z)|^p(1-|z|)^{p-2}\widehat{\om}\left(z\right)\,dA(z)\\
	&\le M^p_\infty\left(\frac{K}{K+1},f'\right)+\int_{0}^1\Psi_{f',M,p}(re^{i\theta})(1-r)^{p}\widetilde{\om}\left(r\right)\,dr.
	\end{split}
	\end{equation*}
Since $r\mapsto\Psi_{f',M,p}(re^{i\theta})$ is non-decreasing by \eqref{Eq:increasing}, we may proceed as in the case $1<p<\infty$. Namely, \cite[Lemma~9(v)(vi)]{PGRV2024} imply \eqref{lejherlkhg} with $\lesssim$ in place of $\le$, and hence an application of \eqref{trivial} together with Lemma~\ref{Prop2.3} and Theorem~\ref{th:tentmaximal} imply 
		\[
		\|f\|_{AL_{p}^{q}(\omega)}
		\lesssim\|N(f')(1-|\cdot|)\|_{T_{p}^{q}(\omega)}
		\lesssim\|f'(1-|\cdot|)\|_{T_{p}^{q}(\omega)},\quad f\in\H(\D),\quad f(0)=0.
    \]
Since $\om_{[p]}\in\DDD$ by the hypothesis $\om\in\DDD$, Theorem~\ref{th:samespaceintro 1} yields
		\[
		\|f\|_{AL_{p}^{q}(\omega)}
		\lesssim\|f'(1-|\cdot|)\|_{L_{p}^{q}(\omega)},\quad f\in\H(\D),\quad f(0)=0.
    \]
The general case
 \[
		\|f\|_{AL_{p}^{q}(\omega)}
		\le C\left(\|f^{(n)}(1-|\cdot|)^{n}\|_{L_{p}^{q}(\omega)}+\sum_{j=0}^{n-1}|f^{(j)}(0)|\right),\quad f\in\H(\D),
    \]
concerning the average radial integrability spaces 
 follows by iterating this estimate.  The statement of the theorem follows from Theorem~\ref{th:samespace2} because
$\om\in\DDD$.
\end{proof}

\begin{Prf}{\em{Theorem~\ref{thm:LPchar}.}} 
Assume that   $\omega \in \mathcal{D}$.  Then,  (ii) and (iii) hold by  Theorems~\ref{theo_LP_1} and \ref{Thm:LP}. Conversely, if (ii) or (iii) holds then $\om\in\DD$ by 
Theorem~\ref{theo_LP_1}. Therefore,  it suffices to show that whenever $\om\in\DD$ both statements (ii) and (iii) guarantee  $\om\in\DDD$. By
 testing on the monomials $f_n(z)=z^n$,  we get
$${\om}_{np+1}\asymp n^{kp} (\om_{[kp]})_{np+1}, \quad n\in\N\cup\{0\}.$$
Then,  $\om\in \DDD$     by \cite[(1.2) and Theorem~3]{PR19}.
\end{Prf}

\section{Integration operators}\label{S4}

We begin with constructing suitable test functions. For that purpose we need the following lemma.

\begin{lemma}\label{le:test}
Let $0<p,q<\infty$ and $\omega\in\widehat{\mathcal{D}}$. Further, let $\beta_0=\beta_0(\omega)$ be that of Lemma~\ref{Lem_D_hat}(iii), and $\beta>\frac{1}{q}+\frac{\beta_{0}}{p}$. Then the function
    \[
    f_\lambda(z)=\frac{1}{\left( 1- \overline{\lambda} z \right)^{\beta}},\quad z\in\D,
    \]
satisfies  
		$$
    \|f_{\lambda}\|_{L^q_p(\om)}
		\lesssim\widehat{\omega}(\lambda)^{\frac{1}{p}}(1-|\lambda|)^{\frac{1}{q}-\beta},\quad \lambda \in \D.
    $$
 %  for any $\lambda \in \mathbb{D}$ and $\beta >\frac{1}{q}+ \frac{\beta_{0}}{p}$, where $\beta_0$ is that from  \eqref{eq_D_hat_test_function}.
\end{lemma}

\begin{proof}
It suffices to prove the estimate for $|\lambda|\in \left(\frac{1}{2},1\right)$. Write $r_n=1-2^{-n}$ for all $n\in\N\cup\{0\}$. Then 
Theorem~\ref{th:radialmaximal} and standard estimates imply
		\begin{equation}\begin{split}\label{eq:t11}
		\|f_{\lambda}\|_{L^q_p(\om)}^{q}
		&\lesssim\int_{0}^{1} \left( \int_{\frac{1}{2}}^{1} \frac{\omega(r)}{\left[ (1-r)+(1-|\lambda|)+ \theta\right]^{p\beta}} \, dr
  \right)^{\frac{q}{p}}   \, d \theta\\
   &\le\left(\int_{0}^{1-|\lambda|}+\int_{1-|\lambda|}^1\right)\left(\sum_{n=1}^{\infty} \frac{\widehat{\omega}(r_n)-\widehat{\omega}(r_{n+1})}{\left[ (1-r_{n+1})+(1-|\lambda|)+ \theta\right]^{p\beta}}\right)^{\frac{q}{p}}   \, d \theta\\
	&=I_1(\lambda)+I_2(\lambda),\quad \lambda\in\D.
	\end{split}\end{equation}
Let $N=N(\lambda)\in\mathbb{N}$ such that $r_N\le|\lambda|<r_{N+1}$, that is, $1-r_{N+1}=2^{-N-1}\le1-|\lambda|<2^{-N}=1-r_{N}$. Then
		\begin{equation*}
		\begin{split}
		I_1(\lambda)
		&\le(1-|\lambda|)\left(\left(\sum_{n=1}^{N}+\sum_{n=N+1}^{\infty}\right)
		\frac{\widehat{\omega}(r_n)-\widehat{\omega}(r_{n+1})}{\left[ (1-r_{n+1})+(1-|\lambda|)\right]^{p\beta}}\right)^{\frac{q}{p}}\\
		&\le(1-|\lambda|)\left(\sum_{n=1}^{N}\frac{\widehat{\omega}(r_n)-\widehat{\omega}(r_{n+1})}
		{\left(1-r_{n+1}\right)^{p\beta}}
		+\sum_{n=N+1}^{\infty}\frac{\widehat{\omega}(r_n)-\widehat{\omega}(r_{n+1})}{\left(1-|\lambda|\right)^{p\beta}}\right)^{\frac{q}{p}}\\
		&\le(1-|\lambda|)\left(\sum_{n=1}^{N}2^{(n+1)p\beta}\widehat{\omega}(r_n)
		+\frac{\widehat{\omega}(\lambda)}{\left(1-|\lambda|\right)^{p\beta}}\right)^{\frac{q}{p}},\quad\lambda\in\D.
	\end{split}
		\end{equation*}
An application of Lemma~\ref{Lem_D_hat}(iii) gives $\widehat{\omega}(r_n)\lesssim2^{(N-n)\beta_{0}}\widehat{\omega}(r_N)$ for all $1\le n \le N$, and hence
	\begin{align}\label{elefhgeofj}
	\sum_{n=1}^{N} 2^{(n+1)p\beta}\widehat{\omega}(r_n)
	&\lesssim2^{p\beta}2^{Np\beta}\widehat{\omega}(r_N)\sum_{n=1}^{N}2^{(p\beta-\beta_0)(n-N)}
	%\lesssim2^{Nq\beta}\widehat{\omega}(r_N)
	\lesssim\frac{\widehat{\omega}(\lambda)}{\left(1-|\lambda|\right)^{p\beta}},\quad \lambda\in\D,
	\end{align}
because $p\beta-\beta_0>0$. Thus 
	$$
	I_1(\lambda)\lesssim\widehat{\omega}(\lambda)^{\frac{q}{p}}(1-|\lambda|)^{1-q\beta},\quad \lambda \in \D,
	$$
that is, $I_1(\lambda)$ obeys the upper bound of the statement. 

It remains to deal with $I_2(\lambda)$. To do this, observe that, for $0<\theta<1$, there exists $M=M(\theta)\in\mathbb{N}\cup\{0\}$ such that $r_M \leq 1-\theta < r_{M+1}$, so $1-r_{M+1}=2^{-M-1} \leq \theta < 2^{-M}= 1-r_{M}$. Therefore an argument similar to that applied in \eqref{elefhgeofj} yields
		\begin{align*}
    I_2(\lambda)
		&\le\int_{1-|\lambda|}^{1}
		\left(\sum_{n=1}^{M}\frac{\widehat{\omega}(r_n)-\widehat{\omega}(r_{n+1})}{(1-r_{n+1})^{p\beta}}
		+\sum_{n=M+1}^{\infty}\frac{\widehat{\omega}(r_n)-\widehat{\omega}(r_{n+1})}{\theta^{p\beta}}
		\right)^{\frac{q}{p}}   \, d \theta \\
		&\lesssim\int_{1-|\lambda|}^{1} 
		\left(2^{Mp\beta}\widehat{\omega}(r_M)+\frac{\widehat{\om}(r_{M+1})}{\theta^{p\beta}}\right)^\frac{q}{p} \, d \theta
		\lesssim\int_{1-|\lambda|}^{1}\frac{\widehat{\omega}(1-\theta)^{\frac{q}{p}}}{\theta^{q\beta}}\, d \theta. 
	\end{align*}
Since $1-\theta <|\lambda|$, an application of Lemma~\ref{Lem_D_hat}(iii) yields $\widehat{\omega}(1-\theta)\lesssim \left( \frac{\theta}{1-|\lambda|}\right)^{\beta_{0}} \widehat{\omega}(\lambda)$. By using our choice $\beta>\frac{1}{q}+\frac{\beta_0}{p}$, we then obtain
	\begin{align*}
	I_2(\lambda)
	\lesssim\frac{\widehat{\omega}(\lambda)^{\frac{q}{p}}}{(1-|\lambda|)^{\frac{q}{p}\beta_0}} \int_{1-|\lambda|}^{1}\frac{d\theta}{\theta^{q \beta-\frac{q}{p}\beta_0}}
	\lesssim\widehat{\omega}(\lambda)^{\frac{q}{p}}(1-|\lambda|)^{1-q\beta},\quad \lambda \in \D,
	\end{align*}
and the proof is complete.
\end{proof}

For $g\in\H(\D)$ and $n,k\in\N\cup\{0\}$ such that $0\le k<n$, we define the operator $T_g^{n,k}$ by
		\begin{align*}
    T_g^{n,k}(f)=T_{I}^n \left(f^{(k)} g^{(n-k)}\right),
		\end{align*}
where $T^n_I=T_I\circ\cdots\circ T_I$ refers to the $n$ compositions of the integral operator induced by the identity mapping $I$.

\begin{proposition}\label{pr:Suf_Tgk_bound}
Let $0< p,q<\infty$, $n\in\N$ and $k\in\N\cup\{0\}$ such that $0\le k<n$, and $\omega\in \mathcal{D}$. If $g\in\mathcal{B}$, then $T_g^{n,k}: AT_p^q(\omega)\rightarrow AT_p^q(\omega)$ is bounded.
\end{proposition}

\begin{proof}
Theorems~\ref{theo_LP_1} and~\ref{Thm:LP} imply
\begin{align*}
\|T_g^{n,k}(f)\|_{AT_p^q(\omega)}
&\lesssim \|f^{(k)} g^{(n-k)} (1-|z|)^{n}\|_{T_p^q(\omega)}\lesssim \|g\|_{\mathcal{B}} \|f^{(k)}(1-|z|)^{k}\|_{T_p^q(\omega)}\\
&\lesssim \|g\|_{\mathcal{B}} \|f\|_{T_p^q(\omega)},
\end{align*}
and thus the assertion is proved.
\end{proof}

\begin{Prf}{\em{Theorem~\ref{thm:Tga_bound}.}}
If $g\in\mathcal{B}$, then 
$T_{g,a}=T^{{n,0}}_g+\sum_{k=1}^{n-1}a_kT_g^{n,k}$ is bounded by Proposition~\ref{pr:Suf_Tgk_bound}.

Conversely, assume that $T_{g,a}:AT_p^q(\omega)\rightarrow AT_p^q(\omega)$ is bounded. Let $\beta_0=\beta_0(\om)>0$ be that of Lemma~\ref{Lem_D_hat}(iii). Further, for $\lambda\in\D$ and $\gamma>\frac{1}{p}+\frac{\beta_0}{q}$, let
	\begin{align*}
	f_{\lambda}(z)
	=\frac{(1-|\lambda|^2)^{\gamma-\frac{1}{p}-\frac{1}{q}}}{(1-\overline{\lambda}z)^\gamma},\quad z\in\D.
	\end{align*}
Then Lemmas~\ref{Prop2.3} and~\ref{le:test} yield
	\begin{align*}
	|(T_{g,a}(f_{\lambda}))^{(n)}(\lambda)|
	\lesssim\frac{\|T_{g,a}(f_\lambda)\|_{T_p^q(\omega)}}{\widehat\omega(\lambda)^{\frac{1}{p}}(1-|\lambda|^2)^{n+\frac{1}{q}}}
	\lesssim\frac{\|T_{g,a}\|_{T_p^q(\omega)\to T_p^q(\omega)}}{(1-|\lambda|^2)^{n+\frac{1}{p}+\frac{1}{q}}}.
	\end{align*}
Hence there exists a constant $C=C(\gamma)>0$ such that 
\begin{align*}
\left|\sum_{k=0}^{n-1}  \frac{a_k \overline{\lambda}^{k} (\gamma)_k}{(1-|\lambda|^2)^{k+\frac{1}{p}+\frac{1}{q}}} g^{(n-k)}(\lambda) \right|  \leq C(\gamma) \frac{\|T_{g,a}\|_{T_p^q(\omega)\to T_p^q(\omega)}}{(1-|\lambda|^2)^{n+\frac{1}{p}+\frac{1}{q}}},
\end{align*}
where $(\gamma)_k=\gamma(\gamma+1)\cdots (\gamma+k-1)$, $k\ge 1$,  and $(\gamma)_0=a_0=1$. By re-arranging factors, it follows that
\begin{align*}
\left|\sum_{k=0}^{n-1}  a_k \overline{\lambda}^{k} (\gamma)_k (1-|\lambda|^2)^{n-k} g^{(n-k)}(\lambda) \right|  \leq C(\gamma) \|T_{g,a}\|_{T_p^q(\omega)\to T_p^q(\omega)},
\end{align*}
 and hence  $g\in\mathcal{B}$  by  \cite[Lemma~2.3]{Chalmoukis}.
\end{Prf}

\section{Bergman Projection}\label{S5}

In our proof of Theorem~\ref{BergmanProj1} we will first show that (iv)$\Rightarrow$(ii). With this aim let us observe that the condition $B_p(\gamma,\omega)<\infty$ implies $L^q_p(\omega)\subset L^1_\gamma=\{f: \int_{\D}|f|v_\gamma dA<\infty\}$. In particular, 
$P^+_{\gamma}(f)\in \H(\D)$ is well defined for each $f\in L^q_p(\omega)$. 

Throughout the proof of Theorem~\ref{BergmanProj1}, we will estimate  the kernel 
	$$
	K_\gamma(re^{i\theta},\rho e^{i\varphi})=\frac{2\rho(1-\rho^2)^\gamma}{|1-r\rho e^{i(\theta-\varphi)}|^{\gamma+2}}
	$$
from above by using a discrete positive kernel which is more convenient for our purposes. In order to do so,    
 we introduce some necessary terminology. 
Given a measure space $(X,\mathcal{M},\mu)$ and a   $\mu$-measurable function $M: X\times X\mapsto \C$, we consider the kernel integral operator
	\begin{align*}
	T_{M}(f)(z)=\int_{X}f(u)M(z,u) \, d\mu(u),\quad  z\in X,
	\end{align*}
whenever such integral is well defined.

From now on, with a little abuse of notation, $|\theta-\varphi|$ will denote the distance between $\theta$ and $\varphi$ in the quotient group $\R/2\pi \Z$, that is, $\min_{k\in \Z} |\theta-\varphi+2k\pi|.$ Since $L^q_p(\omega)\subset L^1_\gamma$ whenever $B_p(\gamma,\omega)<\infty$, to prove the boundedness of $P^+_\gamma$ on  $L^q_p(\omega)$, it is sufficient to establish the boundedness of $T_{{\tilde K}_\gamma}$, where 
	$$
	\tilde{K}_\gamma(re^{i\theta},\rho e^{i\varphi})
	=\rho(1-\rho)^{\gamma}\left|1-re^{i\theta}\overline{\rho e^{i\varphi}}\right|^{-2-\gamma}\chi_{\{|\theta-\varphi|\leq 1\}}
	\cdot\chi_{\{\min\{r,\rho\}> \frac{1}{2} \}}.
	$$ 
Next, observe that 
\begin{align*}
\tilde{K}_\gamma(re^{i\theta},\rho e^{i\varphi})
&= \rho\left((1-r\rho)^{2}+2r\rho(1- \cos(\theta-\varphi))\right)^{-1-\frac{\gamma}{2}}(1-\rho)^{\gamma}\chi_{\{|\theta-\varphi|\leq 1\}}\cdot\chi_{\{\min\{r,\rho\}> \frac{1}{2} \}}\\
&\lesssim \rho\left((1-r\rho)^{2}+r\rho|\theta-\varphi|^{2}\right)^{-1-\frac{\gamma}{2}}(1-\rho)^{\gamma}\chi_{\{|\theta-\varphi|\leq 1\}}\cdot\chi_{\{\min\{r,\rho\}> \frac{1}{2} \}} \lesssim D(\theta,\varphi,r,\rho),
\end{align*}
where 
\begin{align*}
D(\theta,\varphi,r,\rho)=
\begin{cases}
0, & \mbox{if} \quad |\theta-\varphi|> 1 \quad \text{or}\quad \max\{r,\rho\}\le \frac{1}{2},\\
\frac{\rho(1-\rho)^\gamma}{|\varphi-\theta|^{2+\gamma}}, & \mbox{if} \quad 1\geq |\theta-\varphi|\geq 1-r\rho\quad\text{and}\quad \min\{r,\rho\}> \frac{1}{2},  \\
\frac{\rho(1-\rho)^\gamma}{(1-r\rho)^{2+\gamma}}, & \mbox{if} \quad |\theta-\varphi|\leq 1-r\rho \quad\text{and}\quad \min\{r,\rho\}> \frac{1}{2}.\\
\end{cases}
\end{align*}
The change of variables $x=1-r$ and $y=1-\rho$ now yields 
	$$
	\frac{\tilde{H}(\theta,\varphi,x,y)}{2^{2+\gamma}}\leq D(\theta,\varphi,1-x,1-y)
	\le\tilde{H}(\theta,\varphi,x,y),\quad x,y\in[0,1],\quad \theta,\varphi \in [0,2\pi],
	$$ 
with
\begin{align*}
\tilde{H}(\theta,\varphi,x,y)=
\begin{cases}
0, & \mbox{if} \quad |\theta-\varphi|> 1\quad   \text{or}\quad   \min\{x,y\}\ge  \frac{1}{2} \\
\frac{y^{\gamma}(1-y)}{|\theta-\varphi|^{2+\gamma}}, & \mbox{if} \quad 1\geq |\theta-\varphi|\geq \max\{x,y\}  \quad\text{and}\quad  \max\{x,y\}<  \frac{1}{2}, \\
\frac{y^{\gamma}(1-y)}{(\max\{x,y\})^{2+\gamma}}, & \mbox{if} \quad \frac{1}{2}>\max\{x,y\}\geq |\theta-\varphi|,\\
\end{cases}
\end{align*}
because $\max\{x,y\}\leq 1-r\rho\leq 2\max\{x,y\}$. Consequently, in view of the above inequalities, in order to prove that the maximal Bergman projection $P^+_{\gamma}$  is bounded on $L^q_p(\omega)$, it is enough to prove that the operator $T_{\tilde{H}}$ defined on $[0,2\pi)\times(0,1)$ is bounded on 
 $L^{q}_p\left(\omega_L\,dx\,d\theta, [0,2\pi)\times(0,1)\right)$,
where $\omega_L(x)=(1-x)\omega(1-x)$.
In that case we have $\| P^+_{\gamma}\|\lesssim1+\| T_{\tilde{H}}\|$.
With this aim, let us define the sets
\begin{align*}
J_n:=\left\{(\theta,\varphi,x,y)\in  [0,2\pi)^{2}\times [0,1)^2: \max\{x,y\}\leq 2^{n}|\theta-\varphi|< 2\max\{x,y\}< 1\right\}, \quad n\in\N\setminus\{0\},
\end{align*}
and $J_0:=\left\{(\theta,\varphi,x,y): \max\{x,y\}\leq |\theta-\varphi|\leq 1, \, \max\{x,y\}< \frac{1}{2} \right\}$.
Notice that $J_n\cap J_m = \emptyset$ for $m\neq n$. A calculation shows that
\begin{equation}\label{eq:discretekernel}
\tilde{H}(\theta,\varphi,x,y)\asymp \frac{y^\gamma(1-y)}{|\theta-\varphi|^{\gamma+2}}\sum_{n\geq 0}\frac{\chi_{J_n}(\theta,\varphi,x,y)}{2^{n(\gamma+2)}}. 
\end{equation}

For a given radial weight $\nu$, we define $\dot{\nu}(s)=s\nu(s)$ and the maximal operator
$$W_{\nu}f(x)=\begin{cases}
	\sup_{1\geq t\geq x} \frac{ \int_{0}^{t} f(u) (1-u)\nu(1-u) du}{\widehat{\dot{\nu}}(1-t)}, \ &\mbox{ if } 0<x\le 1,\\
	0,\ &\mbox{ if } x> 1.
	\end{cases}$$

For each $\theta\in\mathbb{R}$ and a function $f$ defined on $[0,2\pi)\times[0,1)$, we denote $f_\theta(x)=f(\theta,x)$.

\begin{lemma}\label{PRineq}
Let $\om$ be a radial weight. Then
	\begin{align*}
	&\int_{0}^{2\pi}\hspace{-0.5em} \int_{0}^{1} (T_{\tilde H}f)(\theta,x) g(\theta,x) \omega_L(x)  dx\ d\theta
\\ & \lesssim \sum_{n\geq 0} \int_{0}^{2\pi}\int_{0}^{2\pi} \widetilde{\dot{\omega}}(1-2^n|\theta-\varphi|) W_{v_\gamma}f_\varphi( 2^{n}|\varphi-\theta|) W_{\omega}g_\theta( 2^{n}|\varphi-\theta|)\ d\theta\ d\varphi
	\end{align*}
for any pair $(f,g)$ of positive measurable functions on $[0,2\pi)\times[0,1)$.
\end{lemma}

\begin{proof}
By using the definition of the kernel $\tilde{H}$, \eqref{eq:discretekernel} and  grouping terms, it follows that
	\begin{align*}
	&\int_{0}^{2\pi}\hspace{-0.5em} \int_{0}^{1} T_{\tilde{H}}f(\theta,x)\ g(\theta,x) \omega_L(x) dx\ d\theta\\
	&=\int_{0}^{2\pi} \hspace{-0.5em}\int_{0}^{2\pi} \hspace{-0.5em}\int_{0}^{1}\hspace{-0.5em}\int_{0}^{1} \tilde H(\theta,\varphi,x,y) f(\varphi,y) g(\theta,x)\omega_L(x) dx\ dy\ \frac{d\varphi}{2\pi}\ d\theta\\
	&\lesssim\sum_{n\geq 0} 2^{-n(\gamma+2)}\iiiint_{A_n } \frac{y^\gamma(1-y)}{|\theta-\varphi|^{\gamma+2}} f(\varphi,y) g(\theta,x)\  \omega_L(x) dx\ dy\ d\varphi\ d\theta,
\end{align*}
where $A_n=\{(\theta,\varphi,x,y)\in  [0,2\pi)^{2}\times [0,1)^2: 0\le x\leq  2^n|\theta-\varphi|\leq 1\,\text {and}\,  0\le y\leq  2^n|\theta-\varphi|\leq 1  \}$.
Consequently,
%bearing in mind Lemma~\ref{Lem_D}(iii)
\begin{align*}
	&\int_{0}^{2\pi}\hspace{-0.5em} \int_{0}^{1} T_{\tilde{H}}f(\theta,x)\ g(\theta,x)\ \omega_L(x) dx\ d\theta\\
	&\lesssim \sum_{n\geq 0}\iint_{|\varphi-\theta|\leq 2^{-n}} \widetilde{\dot{\omega}} (1-2^{n}|\theta-\varphi|)\left(\frac{1}{2^{n(\gamma+1)}|\theta-\varphi|^{\gamma+1}}\int_{0}^{2^{n}|\theta-\varphi|} f(\varphi,y)\ y^{\gamma}(1-y) dy\right)\\
	& \times\left(\frac{1}{\widehat{\dot{\omega}}(1-2^n|\theta-\varphi|)}\int_{0}^{2^n|\theta-\varphi|} g(\theta,x) \omega_L(x) dx\right)\ d\theta\ d\varphi\\
	&\lesssim \sum_{n\geq 0}\int_{0}^{2\pi}\int_{0}^{2\pi} \widetilde{\dot{\omega}}(1-2^n|\theta-\varphi|)  W_{v_\gamma}f_\varphi( 2^{n}|\varphi-\theta|) W_{\omega}g_\theta( 2^{n}|\varphi-\theta|)\, d\theta\ d\varphi,
	\end{align*}
and we are done.
\end{proof}

For each $a\in\D$, let $I_a=\left\{e^{i\theta}:|\arg(a e^{-i\theta})|\le \frac{(1-|a|)}{2}\right\}$. The Carleson square induced by $a\in\D$ is the set 
	$$
	S(a)=\{z\in\D: |z|\ge |a|,\, e^{it}\in I_a\}.
	$$
Further, for each weight $\omega$ and $\varphi\in L^1_\om$, the H\"ormander-type maximal function is 
	$$
  M_{\om}(\varphi)(z)=\sup_{z\in S}\frac{1}{\om\left(S\right)}\int_{S}|\varphi(\z)|\om(\z)\,dA(\z).
	$$

\begin{lemma}\label{HLgamma}
Let $-1<\gamma<\infty$, $1<p<\infty$ and let $\omega$ be radial weight such that $B_p(\gamma,\omega)<\infty$. Then $W_{v_\gamma}$ is a bounded operator on $L^{p}(\omega_L, (0,1])$.
\end{lemma}

\begin{proof}
For a radial weight $\eta$ and a radial function $\varphi$, a straightforward calculation shows that
	$$
	M_{\eta}(\varphi)(1-|z|)=W_\eta(\phi)(|z|), \quad z\in\D\setminus\{0\},
	$$
where $\varphi(|z|)=\phi(1-|z|)$. Therefore, it is enough to prove that the  H\"ormander-type maximal function 
$M_{v_\gamma}$ is bounded on $L^p(\omega,\D)$. But this immediately follows from \cite[(4.7)]{PottRegueraJFA}.
\end{proof}

\begin{lemma}\label{HLw}
Let $1<p<\infty$ and $\omega\in\DD$. Then $W_{\omega}$ is a bounded operator on $L^{p}(\omega_L, (0,1])$.
\end{lemma}

\begin{proof}
Since  $\omega\in\DD$, the H\"ormander-type maximal function $M_{\omega}$ is bounded on $L^p(\omega,\D)$ by \cite[Theorem~3.4]{PelSum14}. Therefore
by arguing as in the proof of Lemma~\ref{HLgamma}, the assertion follows.
\end{proof}

\begin{proof}[Proof of Theorem~\ref{BergmanProj1}]
We will show that (iv)$\Rightarrow$(ii)$\Rightarrow$(i)$\Rightarrow$(iii)$\Leftrightarrow$(iv), (i)$\Leftrightarrow$(v) and (iv)$\Rightarrow$(vi)$\Rightarrow$(vii)$\Rightarrow$(iii).

Assume first (iv). By using \cite[Theorem 1]{Benedek_Panzone} and our previous considerations it is enough to prove that
	\begin{equation}\label{eq:proj1}
	\sup\int_{0}^{2\pi}\hspace{-0.5em} \int_{0}^{1} (T_{\tilde H}|f|)(\theta,x)\ |g|(\theta,x) \omega_L (x) dx\ d\theta<\infty,
	\end{equation}
where the supremum is taken over all the pairs of functions $(f,g)$ such that 
	$$
	\|f\|_{ L^q_p \left(\omega_L \,dx\,d\theta, [0,2\pi)\times(0,1)\right) }
	\le1\quad\text{ and}\quad \|g\|_{ L^{q'}_{p'} \left(\omega_L \,dx\,d\theta, [0,2\pi)\times(0,1)\right)}\le1.
	$$
For such $f$ and $g$, write $F=W_{v_\gamma}|f|$ and $G=W_{\omega}|g|$. Fix $K>1$, and let $\{\rho_n\}$ be the sequence defined in \eqref{eq:rhon} in terms of $\omega$ and $K$. Consider the sequences of functions $f_{n}(\varphi)=F(\varphi,1-\rho_n)$ and $g_{n}(\varphi)=G(\varphi,1-\rho_n)$ for $\varphi\in\mathbb{T}$ and $n\in \N\cup\{0\}$. Notice that for all $x\in I_{n}=[1-\rho_{n},1-\rho_{n-1})$ we have $f_{n-1}(\varphi)\leq F(\varphi,x)\leq f_{n}(\varphi)$ and $g_{n-1}(\varphi)\leq G(\varphi,x)\leq g_{n}(\varphi)$. Hence
	\begin{align*}
	\sum_{n=1}^{\infty} f_{n-1}(\varphi) \chi_{I_n}(x)&\leq F(\varphi,x)\leq \sum_{n=1}^{\infty} f_{n}(\varphi) \chi_{I_n}(x),\\
	\sum_{n=1}^{\infty} g_{n-1}(\varphi) \chi_{I_n}(x)&\leq G(\varphi,x)\leq \sum_{n=1}^{\infty} g_{n}(\varphi) \chi_{I_n}(x).
	\end{align*}
Then, by applying Lemma~\ref{HLgamma} in the inner integral to each function $f_\varphi$, we obtain 
	\begin{equation}
	\begin{split}\label{bsequ}
	\int_{0}^{2\pi} \left(\sum_{n=1}^{\infty} f_{n-1}^{p}(\varphi)\ K^{-n}\right)^{\frac{q}{p}}\ d\varphi
	&\lesssim \int_{0}^{2\pi}\left(\int_{0}^{1}|F(\varphi,x)|^{p}\omega_L(x)\,dx\right)^{\frac{q}{p}}\,d\varphi \\
	&\lesssim \int_{0}^{2\pi}\left(\int_{0}^{1}|f(\varphi,x)|^{p}\omega_L(x)\,dx\right)^{\frac{q}{p}}\,d\varphi<\infty.
	\end{split}
	\end{equation}
By \cite[Lemma~9(i)]{PRR23}, Lemma~\ref{HLw}, and repeating the argument above we get 
	\begin{align}\label{bsequ1}
	\int_{0}^{2\pi} \left(\sum_{n=0}^{\infty}g_{n}^{p'}(\varphi) K^{-n}\right)^{\frac{q'}{p'}}\ d\varphi < \infty.
	\end{align}
Let us prove next that the hypothesis $$B_p(\gamma,\omega)=\sup\limits_{0\le r<1}\frac{\left(\int_r^1 \omega(t)t\,dt\right)^\frac1p\left(\int_r^1 \sigma(t)t\,dt\right)^\frac1{p'}}{\int_r^1  v_{\gamma}(t)t\,dt}<\infty$$ implies $\omega\in\Dd$. Since $v_{\gamma}\in\Dd$, there exists $K>1$ and $C>0$ such that 
	$$
	\int_r^1  v_{\gamma}(t)t\,dt\le C \int_r^{1-\frac{1-r}{K}}  v_{\gamma}(t)t\,dt,\quad 0\le r<1.
	$$
Therefore
	\begin{equation*}
	\begin{split}
	\left(\int_r^1 \omega(t)t\,dt\right)^\frac1p\left(\int_r^{1-\frac{1-r}{K}}\sigma(t)t\,dt\right)^\frac1{p'} & \le 
	\left(\int_r^1 \omega(t)t\,dt\right)^\frac1p\left(\int_r^1 \sigma(t)t\,dt\right)^\frac1{p'}\\
	&\le B_p(\gamma,\omega)\int_r^1  v_{\gamma}(t)t\,dt\\
	&\le C  B_p(\gamma,\omega)\int_r^{1-\frac{1-r}{K}}  v_{\gamma}(t)t\,dt\\
	&\le C  B_p(\gamma,\omega)\left(\int_r^{1-\frac{1-r}{K}} \omega(t)t\,dt\right)^\frac1p\left(\int_r^{1-\frac{1-r}{K}} \sigma(t)t\,dt\right)^\frac1{p'}
	\end{split}
	\end{equation*}
for all $0\le r<1$, that is, $\int_r^1 \omega(t)t\,dt \lesssim \int_r^{1-\frac{1-r}{K}} \omega(t)t\,dt$ for all $0\le r<1$.
It follows that $\om\in\Dd$.

Next, by Lemmas~\ref{PRineq} and~\ref{Lem_D_check}(ii), we have 
	\begin{align*}
	&\int_{0}^{2\pi} \hspace{-0.5em} \int_{0}^{1} T_{\tilde H}|f|(\theta,x) |g|(\theta,x)\ \omega_L(x)\ dx\ d\theta
\\ &\lesssim \sum_{j\geq 0} \int_{0}^{2\pi} \hspace{-0.5em}\int_{0}^{2\pi}  \widetilde{\dot\omega}(1-2^j|\theta-\varphi|) F(\theta, 2^j|\varphi-\theta|) G(\varphi, 2^j|\varphi-\theta|)\ d\theta\ d\varphi\\
	&\lesssim \sum_{j\geq 0}\int_{0}^{2\pi} \left(\sum_{n=1}^{\infty} \frac{K^{-n}}{1-\rho_{n}} f_{n}(\theta) \int_{0}^{2\pi} g_{n}(\varphi) \chi_{I_n}(2^j|\theta-\varphi|)\ d\varphi\right)d\theta\\
	&\lesssim \sum_{j\geq 0}\int_{0}^{2\pi} \sum_{n=1}^{\infty} f_{n}(\theta)\ 2^{-j} K^{-n}\left( \frac{1}{2^{1-j}(1-\rho_{n-1})}\int_{\theta-2^{-j}(1-\rho_{n-1})}^{\theta+2^{-j}(1-\rho_{n-1})} g_{n}(\varphi)\ d\varphi\right)\ d\theta\\
	&\ \lesssim \sum_{j\geq 0} 2^{-j} \int_{0}^{2\pi} \sum_{n=1}^{\infty} f_{n}(\theta)\ K^{-n} Mg_{n}(\theta)\ d\theta,
	\end{align*}
where the constants involved in the above inequality depend on $\omega$ and $K$. By Hölder's inequality, the Fefferman-Stein vectorial maximal theorem \cite[Theorem~1, p.~107]{FefSt} and the inequalities \eqref{bsequ} and \eqref{bsequ1} we deduce
	\begin{align*}
	&\int_{0}^{2\pi} \sum_{n=1}^{\infty} f_{n}(\theta)\ K^{-n} Mg_{n}(\theta)\ d\theta\\
	&\quad\lesssim  \int_{0}^{2\pi} \left(\sum_{n=1}^{\infty} f_{n}^{p}(\theta)\ K^{-n}\right)^{1/p}\left(\sum_{n=1}^{\infty} (Mg_{n})^{p'}(\theta)\ K^{-n}\right)^{1/p'}\ d\theta\\
	&\quad\leq \left(\int_{0}^{2\pi }\left(\sum_{n=1}^{\infty} f_{n}^{p}(\theta)\ K^{-n}\right)^{q/p}\ d\theta\right)^{1/q}\left(\int_{0}^{2\pi }\left(\sum_{n=1}^{\infty} (Mg_{n})^{p'}(\theta)\ K^{-n}\right)^{q'/p'}\ d\theta\right)^{1/q'}\\
    &\quad\lesssim  \left(\int_{0}^{2\pi }\left(\sum_{n=1}^{\infty} g_{n}^{p'}(\theta)\ K^{-n}\right)^{q'/p'}d\theta\right)^{1/q'}
     <\infty.
	\end{align*}
Therefore \eqref{eq:proj1} holds, and this finishes the proof of (iv)$\Rightarrow$(ii). 

It is clear that (ii)$\Rightarrow$(i), and by mimicking the proof of \cite[Proposition~8(i)]{PRR23} we get (i)$\Rightarrow$(iii). Further, (iii)$\Leftrightarrow$(iv) follows by \cite[Lemma~9(ii)]{PRR23}. Therefore we have proved that the first four conditions are equivalent.

Now we will prove (i)$\Leftrightarrow$(v). The proof is standard, but we provide the details for the convenience of the readers.
Assume (i). For each $g\in AL^{q'}_{p'}(\sigma)$, consider the linear functional $L_g(f)=\int_{\D}f\bar{g}\, v_\gamma\,dA$. Two applications of H\"older's inequality yield $\|L_g\|_{(AL^q_p(\omega))^\star }\le \|g\|_{AL^{q'}_{p'}(\sigma)}$. Take $L\in (AL^q_p(\omega))^\star$. By the Hanh-Banach theorem $L$ can be extended to a bounded linear functional $\widetilde{L}$ on $L^q_p(\omega)$ such that $\|L\|_{(AL^q_p(\omega))^\star }=\|\widetilde L\|_{(L^q_p(\omega))^\star }$. Now, by \cite[Theorem~1]{Benedek_Panzone} and a straightforward calculation there exists $h\in L^{q'}_{p'}(\sigma)$ such that $\widetilde L=L_h$ and $\|\widetilde L\|_{(L^q_p(\omega))^\star }= \| h \|_{L^{q'}_{p'}(\sigma)}$.
Moreover, since $P_\gamma$ is bounded on $L^q_p(\omega)$, by the symmetry of the condition $B_p(\gamma,\omega)<\infty$, $P_\gamma$ and $P^+_\gamma$ are both bounded on $L^{q'}_{p'}(\sigma)$. So, by Fubini's theorem,
	$$
	L(f)=\widetilde{L}(f)=\langle f, h\rangle_{L^2_\gamma}=\langle P_\gamma(f), h\rangle_{L^2_\gamma}=\langle f, g\rangle_{A^2_\gamma},\quad f\in AL^q_p(\omega),
	$$
where $g=P_\gamma(h)$ and $\|g\|_{L^{q'}_{p'}(\sigma)}=\|P_\gamma(h)\|_{L^{q'}_{p'}(\sigma)}\le \|P_\gamma\| \|h\|_{L^{q'}_{p'}(\sigma)}= \|P_\gamma\|
\|L\|_{(AL^q_p(\omega))^\star }$. Therefore $(AL^q_p(\omega))^\star\simeq AL^{q'}_{p'}(\sigma)$, with equivalence of norms, via the $A^2_\gamma$-pairing.

Conversely, assume (v). Take $h \in L^{q'}_{p'}(\sigma)$ and consider the bounded linear functional $L_h(f)=\langle f, h\rangle_{L^2_\gamma}$ on  
$L^q_p(\omega)$ such that $\| L_h\|_{(AL^q_p(\omega))^\star}\le \|h\|_{L^{q'}_{p'}(\sigma)}$. By Fubini's theorem,
$L_h(f)=\langle f,  P_\gamma(h)\rangle_{A^2_\gamma}$ for each polynomial $f$. Further, by the hypothesis, there exists
$g\in AL^{q'}_{p'}(\sigma)$ such that $L_h(f)=\langle f,  g\rangle_{A^2_\gamma}$ for all $f\in AL^q_p(\omega)$ and $\| L_h\|_{(AL^q_p(\omega))^\star } \simeq \|g\|_{AL^{q'}_{p'}(\sigma)}$. 
Now, by testing with the monomials $\{z^n\}_{n\in\N\cup\{0\}}$ we get $g=P_\gamma(h)$. Therefore $\|P_\gamma(h)\|_{L^{q'}_{p'}(\sigma)}\lesssim \|h\|_{L^{q'}_{p'}(\sigma)}$, that is, 
$P_\gamma$ is bounded on $L^{q'}_{p'}(\sigma)$, and thus $P_\gamma$ is bounded on $L^q_p(\omega)$. Therefore (i) is verified.

Next, assume (iv). Then $\om,\sigma\in\DD$ by \cite[Lemma 9(i)]{PRR23}, and hence $AT^q_p(\omega)=  AL^q_p(\omega)$  and  $AT^{q'}_{p'}(\sigma)=AL^{q'}_{p'}(\sigma)$ by Theorem~\ref{th:samespaceintro 1}. Therefore (vi) follows from (v). Conversely, if (vi) holds, the same  argument as above implies (vii). Finally, assume (vii). Then,  by mimicking the proof of \cite[Proposition~8(i)]{PRR23}, we get (iii). With this guidance we consider the theorem proved.
\end{proof}


\begin{thebibliography}{99}

\bibitem{ACPMED2021} 		T.~Aguilar-Hern\'andez, M. D. Contreras and Luis Rodr\'{\i}guez-Piazza,  
												Integration operators in average radial integrability spaces of analytic functions,
												Mediterr. J. Math. 18 (2021), no. 3, Paper No. 117, 41 pp.

\bibitem{ACPJFA2022} 		T.~Aguilar-Hern\'andez, M. D. Contreras and Luis Rodr\'{\i}guez-Piazza,  
												Average radial integrability spaces of analytic functions,
												J. Funct. Anal. 282 (2022), no. 1, Paper No. 109262, 34 pp.

\bibitem{AGJMAA2022} 		T.~Aguilar-Hern\'andez and P.~Galanopoulos, 
												Average radial integrability spaces, tent spaces and integration operators,
												J. Math. Anal. Appl. 523 (2023), no. 2, Paper No. 127028, 38 pp.

\bibitem{Atesis}				H.~Arrousi, 
												Function and Operator theory on large Bergman spaces, PhD Thesis, 2016.

\bibitem{ATTY} 					H.~Arrousi, J.~Taskinen, C.~Tong and Z.~Yuan, 
												Area operators on Large Bergman spaces, preprint. https://tuhat.helsinki.fi/ws/portalfiles/portal/297760749

%\bibitem{arsenovic} 		M. Arsenovic, 
%												Embedding derivatives of M-harmonic functions into Lp-spaces, 
%												Rocky Mountain J. Math. 29 (1999), no. 1, 61--76.

%\bibitem{Bek}           D.~Bekoll\'e,
%                        In\'egalit\'es \'a poids pour le projecteur de Bergman dans la boule unit\'e de
%                        $C^n$,
%                        [Weighted inequalities for the Bergman projection in the unit ball of $C^n$]
 %                       Studia Math. 71 (1981/82), no. 3, 305--323.

\bibitem{BB}            D.~Bekoll\'e and A.~Bonami,
                        In\'egalit\'es \'a poids pour le noyau de
                        Bergman,
                        (French) C. R. Acad. Sci. Paris Ser. A-B 286 (1978), no. 18, 775--778.

\bibitem{Benedek_Panzone} A.~Benedek and R.~Panzone, 
												The spaces $L^p$, with mixed norm, 
												Duke Math. J. 28 (1961), 301--324.

\bibitem{Chalmoukis} 		N. ~Chalmoukis, 
												Generalized integration operators on Hardy spaces, 
												Proc. Amer. Math. Soc. 148 (2020), 3325--3337.

%\bibitem{CMSJFA1985}    R.~R.~Coifman, Y.~Meyer and E.~M.~Stein,
%                        Some new functions spaces and their applications to Harmonic Analysis,
 %                       J. Funct. Anal. 62  (1985), no. 3, 304--335.



%\bibitem{CohnFerRochPLMS01} W.~S.~Cohn, S.~H.~Ferguson and R.~Rochberg,
%                            Boundedness of higher order Hankel forms, Factorization in potential spaces and derivations,
 %                           Proc. London Math. Soc. 82 (2001), no. 1, 110--130.

\bibitem{CMS}           R.~R.~Coifman, Y.~Meyer and E.~M.~Stein,
                        Some new functions spaces and their applications to Harmonic Analysis,
                        J. Funct. Anal. 62  (1985), no. 3, 304--335.

\bibitem{CV}            W.~S.~Cohn and I.~E.~Verbitsky, 
												Factorization of tent spaces and Hankel operators, 
												J. Funct. Anal. 175 (2000), no. 2, 308--329.

\bibitem{Duren1970}     P.~Duren, Theory of $H^p$ Spaces, Academic Press, New York-London 1970.



\bibitem{FefSt}        	C.~Fefferman and E.~M.~Stein, 
												Some maximal inequalities, 
												Amer. J. Math. 93 (1) (1971), 107--115.

\bibitem{HormanderL67}  L.~H\"ormander,
                        $L^p$ estimates for (pluri-)subharmonic functions,
                        Math. Scand. 20 (1967), 65--78.

\bibitem{HLSJFA} 				Z.~Hu,  X.~Lv and A.~Schuster,
												Bergman spaces with exponential weights,
												J. Funct. Anal. 276 (2019), no .5, 1402--1429.

\bibitem{Lu90} 					D.~H.~Luecking, 
												Embedding derivatives of Hardy spaces into Lebesgue spaces,
                        Proc. London Math. Soc. 63 (1991), no.~3, 595--619.

\bibitem{OFJFA1997} 		J.~Ortega and J.~F\'abrega, 
												Holomorphic Triebel-Lizorkin spaces,
												J. Funct. Anal. 151 (1997),  no. 1, 177--212.


\bibitem{Pau-Pelaez:JFA2010}	J.\ Pau and J. A.\ Pel\'aez,
															Embedding theorems and integration operators on Bergman spaces with rapidly decreasing weights, 
															J. Funct. Anal. 259 (2010), no. 10, 2727--2756. 

\bibitem{Pabook2}   		M.~Pavlovi\'c, 
												Function classes on the unit disc. An introduction. 
												De Gruyter Studies in Mathematics, 52. De Gruyter, Berlin, 2014. xiv+449 pp. ISBN: 978-3-11-028123-1.

\bibitem{PelSum14}      J.~A.~ Pel\'aez,
                        Small weighted Bergman spaces,
                        Proceedings of the summer school in complex and harmonic analysis, and related topics, (2016).



\bibitem{PelRat}        J.~A.~ Pel\'aez and J. R\"atty\"a,
												Weighted Bergman spaces induced by rapidly increasing weights,
												Mem. Amer. Math. Soc. 227 (2014), no.~1066.

\bibitem{PelRatMathAnn} J.~A.~ Pel\'aez and J. R\"atty\"a,
                        Embedding theorems for Bergman spaces via harmonic analysis, Math. Ann. 362 (2015), no. 1-2, 205--239.

\bibitem{PR2020}        J. A. Pel\'aez and J. R\"atty\"a,
												Harmonic conjugates on Bergman spaces induced by doubling weights,
												Anal. Math. Phys. 10 (2020), no. 2, Paper No. 18, 22 pp.

\bibitem{PR19} 					J.~A.~ Pel\'aez and J. R\"atty\"a, 
												Bergman projection induced by a radial weight, 
												Adv. Math. 391 (2021), Paper No. 107950, 70 pp.          

\bibitem{PR2023}               J.~A.~ Pel\'aez and J. R\"atty\"a,  Bergman projection and BMO in hyperbolic metric: improvement of classical result.
                                       Math. Z. 305 (2023), no. 2, Paper No. 19, 9 pp.   


\bibitem{PRR23} 				J.~A.~ Pel\'aez, J.~R\"atty\"a and E. de la Rosa,
							Bergman projection on Lebesgue space induced by doubling weight.
                               Results Math. 79 (2024), no.1, Paper No. 27, 19 pp.

\bibitem{PGRV2024}			F. P\'erez-Gonz\'alez, J. R\"atty\"a and T. Vesikko,
												Norm inequalities for weighted Dirichlet spaces with applications to conformal maps, 
												https://arxiv.org/pdf/2311.06191.pdf.

\bibitem{Perala2018} 		A.~Per\"al\"a, 
												Duality of holomorphic Hardy type tent spaces, 
												arXiv :1803 .10584, 2018.

\bibitem{PottRegueraJFA} S.~Pott and M.~Reguera, 
												Sharp B\'ekoll\'e estimates for the Bergman projection,
												J. Funct. Anal. 265 (2013), no. 12, 3233--3244.
 
%\bibitem{Wu2006}  			Z.~Wu,
%												Area operator on Bergman spaces, 
%												Sci. China Ser. A 49 (2006), no. 7, 987--1008.                        

\bibitem{Zhu2007} 			K. Zhu,
												Operator theory in function spaces,
												Math. Surveys Monogr., 138 American Mathematical Society, Providence, RI, 2007, xvi+348 pp.

\end{thebibliography}
\end{document}